\documentclass[11pt]{article}

\usepackage{amssymb,amsmath,amsfonts}
\usepackage{graphicx,color,enumitem}
\usepackage{amsthm} % For proof environment
\usepackage{bm}
\usepackage[round]{natbib}
\usepackage{geometry}

\usepackage[colorinlistoftodos, textwidth=4cm, shadow]{todonotes}
\RequirePackage[colorlinks,citecolor=blue,urlcolor=blue]{hyperref}

\usepackage{bbm}
\usepackage{capt-of}

\newcommand{\E}{{\mathbb E}}

\renewcommand{\P}{{\mathbb P}}

\newcommand{\R}{{\mathbb R}}

\newcommand{\cF}{{\mathcal{F}}}

\newcommand{\CA}{{\mathcal{A}}}

\newtheorem{theorem}{Theorem}

\newtheorem{corollary}[theorem]{Corollary}

\newtheorem{lemma}[theorem]{Lemma}

\newtheorem{remark}[theorem]{Remark}

\theoremstyle{definition}
\newtheorem{example}[theorem]{Example}

\numberwithin{equation}{section}
\numberwithin{theorem}{section}

\definecolor{darkgreen}{rgb}{0,0.7,0}

\newcommand{\iii}{{\vert\kern-0.25ex\vert\kern-0.25ex\vert}}

\usepackage{mathtools}
\mathtoolsset{showonlyrefs}
\usepackage{soul}

\newcommand{\PP}{\mathbb{P}}

\newcommand{\FF}{\mathbb{F}}

\newtheorem{thm}{Theorem}[section]

\newtheorem{lem}[thm]{Lemma}
\newtheorem{pro}[thm]{Proposition}
\newtheorem{defi}[thm]{Definition}
\newtheorem{rem}[thm]{Remark}

\def \exp{ {\rm exp} }
\def \ind{1\!\!1}

\def \cF{{\cal F}}

\def \ind{1\!\!1}

\def \lim{ {\rm lim} }

\let\inf\relax \DeclareMathOperator*\inf{\vphantom{p}inf}

\begin{document}
	
	\title{Gaussian Agency problems with memory and Linear Contracts}
	\author{Eduardo Abi Jaber\thanks{CMAP, Ecole Polytechnique Paris, route de Saclay 91128 Palaiseau Cedex}\\
	St\'ephane Villeneuve\thanks{Universit\'e Toulouse 1 Capitole (TSE-TSM-R), 1 esplanade de l'universit\'e, 31000 Toulouse, stephane.villeneuve@tse-fr.eu. The author acknowledges funding from ANR PACMAN and from ANR under grant ANR-17-EUR-0010 (Investissements d'Avenir program) and gratefully thanks the  FdR-SCOR ``Chaire March\'e des risques et cr\'eation de valeurs".  We are grateful to Thomas Mariotti and Takuro Yamashita for their comments and supports. We thank participants of Financial and Actuarial Mathematics Seminar at Liverpool, Financial and Insurance Mathematics at ETH,Math. Finance Seminar at Bielefeld, 10th AMaMeF Conference at Padova, Annual TSE-SCOR conference at Toulouse for their helpful comments and suggestions.}
	}

	\maketitle
	
\begin{abstract}

	Can a principal still offer optimal dynamic contracts that are linear in end-of-period outcomes when the agent controls a process that exhibits memory?
	We provide a positive answer by considering a general Gaussian setting where the output dynamics are not necessarily semi-martingales or Markov processes. We introduce a rich class of principal-agent models  that  encompasses dynamic agency models with memory. From the mathematical point of view, we develop a methodology to deal with the possible non-Markovianity and non-semimartingality of the control problem, which can no longer be directly solved by means of the usual Hamilton-Jacobi-Bellman equation. Our main contribution is to show that, for one-dimensional models, this setting always allows for optimal linear contracts in end-of-period observable outcomes with a  deterministic optimal level of effort.  In higher dimension, we show that linear contracts are still optimal when the effort cost function is radial and we quantify the gap between linear contracts and optimal contracts for more general quadratic costs of efforts.  
\end{abstract}

\section{Introduction}
The extensive literature analyzing the dynamic principal-agent problem has shown that it is important but difficult to design the optimal shape of contracts in a tractable way. 
Indeed,  optimal contracts in dynamic agency problems are generally defined as complex functionals of a stream of contractible variables, such as revenues. Moreover, as first identified by \citet{rogerson1985repeated}  (see also \cite{laffont2009theory}) theoretical contracts exhibit memory, even in the most commonly used models that assume uncorrelated shocks, which unfortunately prevents them from matching real-world practices (see \citet{bolton2005contract}).  In addition, firms' revenues empirically show long memory\footnote{In this paper, we use indifferently the terms long (short) memory and long (short)-range dependence see definition 2.1. page 42 in \citet{BeranBook}.}    and  we lack a theoretical framework that justifies the signing of simple tractable contracts in an environment with inter-temporal links across time periods.\\ 
In a Brownian setting, the  breakthrough paper by \citet{holmstrom1987aggregation}  (HM)  shows that the optimal contract is linear in profits under some specific assumptions: the agent exerts effort continuously, principal and agent have CARA utilities, the agent bears a pecuniary cost of effort and finally the outcomes generated in the absence of effort are modeled by a fully observable Brownian motion.  Since then, several attempts have been made to obtain closed-form contracts in environments that relax at least one of the assumptions of the HM model.  \citet{sung1995linearity}  showed that the optimal contract is still linear when the agent also controls  the variance of the output. \citet{hellwig2002discrete} showed that linear contracts are nearly optimal in a discrete-time version of the HM model. \citet{edmans2011tractability}  and \citet{edmans2012dynamic} obtain striking general results in a discrete-time model where none of the four hypotheses is retained but where the agent makes its decision in each period after having observed the noise. However, they focus primarily on the cheapest implementation of a particular action, rather than on the objective of maximizing the principal's preference. Another important recent contribution beyond the Holmstrom and Milgrom setting has been made by \citet{carroll2015robustness} who showed that optimal contracts are linear in a general one-period model with uncertainty.\\

In this article, we enrich the Holmstrom and Milgrom modeling framework by going beyond the assumption that the revenues are driven by a Brownian motion.

 We will instead consider Volterra Gaussian processes that are a generalization of the standard Brownian motion to study time-dependent effects. More precisely, a Volterra Gaussian process is a Wiener integral process with respect to a standard Brownian motion involving a deterministic integrand called -{\it kernel}. Thus, at every point in time, it is an infinite linear combination of i.i.d. Gaussian random variables with time-dependent coefficients. Although we are aware of the shortcomings of the three remaining HM assumptions\footnote{For instance, \cite{edmans2011tractability} clearly argue {\it there is ample evidence of decreasing absolute risk aversion, and many effort decisions do not involve a monetary expenditure}}, our targeted choice is primarily motivated by the fact that, by allowing arbitrary integrand kernel functions in the Wiener integral,  our Volterra Gaussian agency model encompasses agency models with short and  long run autocorrelations. In particular, one of the main examples of Volterra Gaussian processes is the mean-reverting process which allows us to get closer of recent models of dynamic contracts with persistence such as those developed in \citet{williams2011ECMA} or career concerns as in \citet{cisternas18}.\\

For a long time, Volterra Gaussian processes have been considered as a natural tool for modeling continuous phenomena with memory.  In particular, the fractional Brownian motion (FBM), a Volterra Gaussian process with short and long range dependence, initially introduced by \citet{kolmogorov1940wienersche}, was  popularized by \citet{mandelbrot1968fractional}  in finance  to model the empirically-validated long-term dependence of stock returns. More recently, a stream of literature suggested the use of variants of the fractional Brownian motion in stochastic volatility modeling to capture the  roughness of the time series of the volatility of an asset which has been observed empirically in the market, see \citet{gatheral2018volatility,AJLP17} and the references therein. In general, such processes are  non-Markovian and non-semimartingales, which make their study more intricate, both theoretically and practically, and  prevents the use of standard stochastic calculus tools. Within the framework of optimal dynamic contracting theory,  the continuous-time semi-martingale models have received a lot of attention the past thirty years, when a significant progress has been made by relying on the recursive approach pioneered by \citet{Green1987}, \citet{spear1987repeated}, \citet{thomas1990income}. Discrete-time models were first developed (see \citet{clementi2006theory,demarzo2007optimal}, followed by \citet{biais2007dynamic}), while the breakthrough paper by \citet{sannikov2008continuous} resulted in the recent study of dynamic contracting in continuous-time models (see also \citet{demarzo2006optimal, biais2010large, demarzo2012dynamic}).
The main advantage of continuous-time semi-martingale models lies in the fact the procedure to find the optimal contract can be embedded in the standard theory of Markovian stochastic control using the theory of martingales and stochastic calculus, see \citet{schattler1993first} for a general presentation and \citet{cvitanic2018dynamic} for a rigorous mathematical justification. Hence, under a fairly general set of assumptions, the contract can be characterized unambiguously by solving an Hamilton-Jacobi-Bellman equation where the so-called agent promised value plays the role of a state variable.\\

 In allowing a non-semimartingale and non-Markovian setting, this paper makes an additional methodological contribution to solve for the optimal contract. The main idea is to use the so-called martingale optimality principle to study the agent and the principal problem sequentially as a Stackelberg game. The first step is to offer a class of incentive-compatible contracts by revisiting the martingale approach of \citet{schattler1993first} and \citet{sannikov2008continuous}.
The second step and our main contribution is to explicitly solve the principal problem which becomes a controlled stochastic Volterra problem. This requires the introduction of an auxiliary state variable - {\it the effort-corrected forward output} - which captures all the non-Markovianity and allows the application of the martingale optimality principle for the principal problem. In a one-dimensional setting, our key result is that the optimal contract is linear in the terminal value of the output, although the principal has in general a coarser information than the agent. The optimal contract has interesting features. The slope or marginal value of the contract is independent of the output dynamics, only the intercept depends on the latter as a function of the optimal effort which is proportional to the kernel. Therefore, random parts of contracts signed in different one-dimensional Gaussian environments are identical although the required effort levels are environment specific, deterministic and exhibit interesting features in relation with the properties of the kernel. \\

  We extend the paper with a discussion of the optimality of linear contracts in higher dimension. We address the multitask principal-agent problem in which a principal with CARA preferences hires a single agent with CARA preferences
  to perform different tasks. The outcome of each task is assumed to follow a Volterra Gaussian process whose evolution depends on the agent's continuous effort in each task while the profit is the sum of these different outcomes. Our main result is that there is no value in observing the agent's activities separately when the cost of effort is assumed to be a radial quadratic function. Under this assumption,  the optimal contract only uses aggregate information and is still linear in the end-of-period outcome. When we consider a general effort cost function, we characterize the contract that would be optimal if the principal were able to observe the Brownian filtration and measure the utility gap when a less-informed principal forces herself  to sign the best linear contract and identify factors that reduce the loss of utility associated with the use of linear contracts. This paper thus shows that linear contracts can closely achieve maximum principal utility in Gaussian environments.

\section{The one-dimensional model}
In this section, we present the economic model which is essentially an extension of the \citet{holmstrom1987aggregation} framework. \\

{\it General Description:} We consider a risk-averse investor, who owns a project and signs a fixed-term contract with a risk-averse manager, the latter being necessary to operate a project.
Time is continuous and the time horizon is $T>0$. In the absence of effort, the  stochastic output process $(X_t)_{t\leq T}$  of the project evolves up to time $T$ as
\begin{equation}\label{eq:Xintro}
X_t=g_0(t)+\int_0^t K(t,s)  dB_s,
\end{equation}
where  $B$ is a standard one-dimensional Brownian motion,  $g_0:[0,T]\to \R$ is a measurable deterministic input function, $K : [0,T]^2 \to \R$ is a measurable Volterra Kernel, i.e.  $K(t,s)=0$ for $s \geq  t$ such that 
\begin{align}
\sup_{t\leq T} \int_0^T K^2(t,s) ds < \infty. 
\end{align}

  We first observe that our model encompasses a large class of output dynamics that offers great modeling flexibility to model agency relationships in different sectors of the economy. Obviously,  this setting contains the Holmstrom and Milgrom Brownian model by choosing $g_0(t)=x_0$ and $K(t,s)=  \sigma$  for any pair  $s < t$ for some constant $\sigma$. Even more generally, the case of a time-dependent volatility can be recovered by setting $K(t,s)= \sigma(s) \ind_{s<t}$ for some square-integrable function $\sigma$. More interestingly, this framework also contains mean-reverting processes which are widely used to model output in the energy and mining sectors, if for instance, we choose $g_0(t)=e^{-\lambda t}x_0+ \frac{\mu_0}{\lambda} (1-e^{-\lambda t})$ and
$K(t,s)=e^{-\lambda(t-s)} \ind_{s<t}$, the output then follows the Ornstein-Ulhenbeck dynamics
$$
dX_t=(\mu_0-\lambda X_t)\,dt + dB_t.
$$
Another example is the Brownian bridge pinned for instance in $0$ at some time $T_0>T$ which falls into this category with a kernel given by
$$
K(t,s)=\frac{T_0-t}{T_0-s}\ind_{s< t}.
$$
The Brownian bridge has the semi-martingale decomposition
$$
dX_t=-\frac{X_t}{T_0-t}\,dt+dB_t,
$$
and may be used in any situation where the agent has access to information about the future output. For example, it can be used to model the output of a seasonal crop which will end up being zero after the harvest season.
A more striking example is given by  the family of fractional Gaussian processes such as:
\begin{itemize}
	\item 
	the Riemann-Liouville fractional Brownian motion  where for $s<t$,
	$$
	K(t,s)=c_H (t-s)^{H-1/2}, \,\quad H \in (0,1),  \mbox{ for some constant } c_H.
	$$
	\item
 the fractional Brownian motion  whose covariance function is $\Sigma_0(s,u)=\frac {1} 2 (s^{2H}+u^{2H}-|s-u|^{2H})$, for $H\in (0,1)$,  admits the Volterra representation \eqref{eq:Xintro} with the kernel	
	\begin{align*}
	K(t,s)= \ind_{s< t}\frac{(t-s)^{H-1/2}}{\Gamma(H+\frac 1 2)} \, {}_2 F_1\left(H-\frac 1 2; \frac 1 2-H; H+\frac 1 2; 1-\frac t s \right),
	\end{align*}
	where ${}_2F_1$ is the Gauss hypergeometric function,
	see  \citet{decreusefond1999stochastic}.
\end{itemize}
 Both types of fractional Brownian motion do not fall into the semi-martingale and Markovian frameworks when the so-called Hurst parameter $H$  is different than $1/2$ (which corresponds to the case of the standard Brownian motion):  they exhibit long range dependence when $H>1/2$ and short run dependence when $H<1/2$  and prove to be statistically very good models for industries related to power generation.\\
  Even more remarkably, equations with delay in the drift in  the form of linear integro-differential convolution equations:  
	\begin{align}\label{eq:stochastic integro}
	dX_t = \left( h(t) + \int_{[0,t]} \mu(ds) X_{t-s} \right) dt + \sigma dB_t , 
	\end{align}
	with initial condition $X_0 \in \R$, where $h:[0,T] \to \R$, $\mu:\mathcal B([0,T]) \to \R$ of bounded variation, $X_0 \in \R$ and $\sigma \in \R$ admit a unique solution in the form of a Gaussian Volterra process \eqref{eq:Xintro}
	for some specific choice of input curve $g_0:[0,T]\to \R$ and convolution kernel $K:[0,T]^2 \to \R$,  see Appendix \ref{Integro} for a detailed presentation of this observation.	For instance,	setting $\mu(dt) = \sum_{k=1}^m a_k \delta_{t_k} $, we recover equations with delay. Such equations fall into the semi-martingale framework, but are clearly not Markovian.\\
While very different in nature and in modeling objective, all these dynamics have in common to be Gaussian processes.

Another important observation about this framework relates to assumptions about the asymmetry in information, which has been interpreted by \cite{holmstrom1987aggregation} as a distinction between linear optimal contracts in outcomes $X$ and those in accounts $B$. We denote by $\FF^B$ the augmented filtration generated by $(B_t)_{t\leq T}$ and $\FF^X$ the one generated by the output process $(X_t)_{t\leq T}$. It readily follows from \eqref{eq:Xintro} that $\FF^X \subset \FF^B$. Hereafter, we assume that the agent has better information than the principal about the project in the sense that he has access to the full information $\FF^B$ while the principal observes only some aggregated information generated by the output $\FF^X$. In general, these two filtrations do not coincide even in one-dimensional models as shown in the following example corresponding to a situation where the principal observes the output in a discretionary way.\\

{\it Discrete observations of a Brownian motion:} Assume $X_t=f(t)B_t$ where $f$ is a bounded function on $[ 0,T ]$. Observe that $X$ is a Volterra process with $K(t,s)=f(t)\ind_{s \le t}$. Consider a subdivision $0<t_1<\ldots<t_n=T$ of the interval $[0,T]$ and let $f$ be the function defined as a linear combination of unit impulses
\begin{align}\label{eq:fdis}
f(t)=\sum_{i=1}^N \ind_{t_i}(t).
\end{align}
The output process is purely discontinuous with $X_{t_i}=B_{t_i}$ and $X_t=0$ for $t \neq t_i$ and may correspond to a situation where the principal performs audits at regular intervals. Therefore, $\FF^X$ is strictly included in $\FF^B$. We deduce that, even in a situation where the principal knows the agent is not exerting effort, the principal has a coarser information than the agent. In Volterra Gaussian models, we must therefore be careful that there may be asymmetric information between the principal and the agent regardless of the agency problem we introduce below.\\

{\it Agency problem:} We assume that the agent can exert a continuous effort $(a_t)_{t\leq T}$ that modifies the probability distribution of $X$ as follows
\begin{equation}\label{outputeffort}
X_t=g_0(t)+\int_0^t K(t,s)(a_s\,ds + dB_s^a),
\end{equation}
where $B^a$ is also a Brownian motion. 
As is customary in the agency theory literature, while the output process $X$ is observable by both players, the effort is the agent's private information. The agent's cost for exercising some effort level $a$ is modeled through a strictly convex $C^2$ function $k(a)$ satisfying $k(0)=0$.  To alleviate the exposition, we will assume hereafter that the effort cost function is quadratic,
\begin{align}\label{eq:costk}
k(a)=\kappa \frac{a^2}{2},  \quad  \mbox{for some } \kappa>0.
\end{align}
Hereafter and in accordance with the paper of \citet{holmstrom1987aggregation}, we model the preferences of the principal and the agent with CARA utility functions that are given respectively by
$$
U_P(x):=-\exp(-\gamma_P x)\, \hbox{ and } U_A(x):=-\exp(-\gamma_Ax), \quad  \forall x\in\R.
$$ 
In the beginning of the relationship, principal and agent agree on a contract of maturity $T$. To foster incentives, the contrat specifies a payment at time $T$ which is modeled by a random variable $\xi$ that is supposed to be $\cF_T^X$ measurable. We assume that both players can fully commit to the contract and that the agent has a reservation utility level $R_0=U_A(y_0)<0$ below which he will refuse the contract. The latter inequality is referred to the participation constraint of the agent who has the option to reject a contract and enjoy a utility of autarky $R_0$.\\

{\it Description of the probabilistic background:} For completeness, we recall the rigorous formulation of the agency problem in order to make understandable the first-order conditions that we will give in the next section. 
Let $(\Omega,\mathcal{F},  \FF :=(\mathcal F_t)_{t\leq T},\PP_0)$ be a filtered probability space on which a $\FF$--Brownian motion $B:=(B_{t})_{{t\leq T}}$  is defined with natural (completed) filtration  $\FF^B:=(\mathcal F^B_t)_{t\leq T}$.

\vspace{0.5em}
The firm's output or cash-flows observed by the principal are given by a stochastic process  $X$ with dynamics under $\PP_0$,
\begin{equation}\label{eq:dynX}
X_t=g_0(t)+\int_0^t K(t,s) d B_s,
\end{equation}
The impact of the agent's effort is modeled as a change of probability measure which changes the drift of the driving Brownian process.  More precisely, agent's  admissible actions are given by the following set  
 \begin{align*}{}
\mathcal A = \left\{ (a_t)_{t \leq T} \,  \FF\mbox{-progressively measurable: there exists $A>0$ s.t.~} \int_0^T a_s^2 ds \leq  A, \; \PP_0-\mbox{a.s.}  \right\}.
\end{align*}

 Observe that the set of admissible actions $\mathcal A$ is not empty because it contains bounded actions. Clearly,  any admissible process $a \in \mathcal A$ satisfies the Novikov's criterion
$$
\E\left[\exp\left( \frac{1}{2} \int_0^T a_s^2ds\right)\right]<+\infty,
$$
 which ensures that  the process $\left(\exp\left( \int_0^T {a_s}\,dB_s-\frac{1}{2}\int_0^T a_s^2\,ds\right)\right)_{0\le t\le T}$ is a martingale, see \citet[Proposition 5.12 p. 198]{karatzasshreve91}.
 We can therefore define  a family of equivalent probability measures $\PP_a$ by
$$
\frac{d\PP_a}{d\PP_0}=\exp\left( \int_0^T {a_s}\,dB_s-\frac{1}{2}\int_0^T a_s^2\,ds\right),
$$
where $a$ ranges  trough $\CA$.
Under $\PP_a$, the process $B^a=B-\int_0^{\cdot} a_s\,ds$ is a $\FF-$Brownian motion by Girsanov theorem and $X$ evolves as 
\begin{equation}\label{eq:dynX:Pa}
X_t=g_0(t)+\int_0^t K(t,s)(a_s\,ds+ d B_s^a).
\end{equation}
Because, the effort is unobservable, the principal only observes the trajectory of the output process $X$, the deterministic curve $g_0$ but not the last two terms of the decomposition \eqref{eq:dynX:Pa} separately. \\
Interestingly,  in the case of general Volterra processes, this model leads to a novel simple setting where we have persistence of past efforts on the output variation. To understand this, let us imagine that the agent makes a constant effort $a$ on the interval $[0,t]$ and then stops exerting effort after $t$. Then,  we have  for $h>0$, that 
$$
\E[(X_{t+h}-X_t)| \mathcal F_t]=g_0(t+h)-g_0(t)+\int_0^t (K(t+h,s)-K(t,s))(dB_s^a+a\,ds),
$$
which induces persistence of past efforts on the future output increments whenever the functions $K(t+h,.)$ and $K(t,.)$ are not identical.
Notice that in the HM model, the kernel is constant, so we recover that past efforts have no influence on future variations of the output.
\\

{\it The Principal-agent problem:} It is well known that principal-agent relationships can be viewed as a Stackelberg game. The principal moves first by offering a contract that consists in a compensation $\xi$, which belongs to the set of $\cF_T^X$ measurable random variables, to the agent. The latter then reacts by choosing an effort policy based on the information available at each date  inducing a probability measure $\PP_a$.
For any given contract $\xi$, let $V_0^A(\xi)$ denote the agent's utility at time $0$ which is defined as
\begin{equation} \label{agentutility}
V_0^A(\xi):=\sup_a \E^a\left( U_A\left(\xi - \int_0^T k(a_s) \,ds \right)\right)
\end{equation}
recall the definition of $k$ in \eqref{eq:costk}.
As common in agency problems, we define the concept of incentive-compatible contracts.
\begin{defi}\label{def:ic}
A contract $\xi$ is said to be {\it incentive compatible} if $V_0^A$ is finite and if  there exists an effort policy $a^*(\xi) \in \CA$ that maximizes \eqref{agentutility}, i.e. 
$$
V_0^A(\xi)=\E^{a^\ast(\xi)}\left( U_A\left(\xi - \int_0^T k(a^*_s(\xi)) \,ds \right)\right).
$$
\end{defi}
It is critical to understand what incentive-compatible contracts are, as these are the ones for which the principal can enforce desirable efforts.
As common in the literature, we will focus on a class $\Xi$ of contracts $\xi$ that are incentive-compatible (IC).  Before defining rigorously the class of IC contracts $\Xi$ we will focus on, we clarify the principal's problem.  By offering an incentive-compatible contract $\xi \in \Xi$, the principal will be able to anticipate the optimal effort level $a^\ast(\xi)$. Hence,  she will propose an incentive-compatible contract that maximizes the expected value of her CARA preference. Then, her aim is to solve
\begin{equation} \label{eq:Principalpb}
V_0^P:=\sup_{\xi \in \Xi} \E^{a^\ast(\xi)}\left[U_P\left(X_T-\xi\right)\right],
\end{equation}
under the participation constraint $\E^{a^\ast(\xi)}\left( U_A\left(\xi - \int_0^T k(a^*_s(\xi)) \,ds \right)\right) \ge R_0$. \\
The first result of this paper is given by the following theorem which shows that the problem \eqref{eq:Principalpb} admits an optimal contract which is linear in end-of-period outcomes. The result of the Holmstrom-Milgrom model thus extends to all Gaussian Volterra processes, even though these may exhibit very different statistical properties.
 Following \citet{schattler1993first}, we introduce the class of contracts we will focus on.  Let us define
 \begin{equation}\label{BSDEdrift}
f^\ast(z):=\frac{\gamma_A}{2}  |z|^2 + \inf_{a \in \mathbb R} \{k(a)-a z\} =\frac{\kappa \gamma_A-1}{2\kappa }z^2 ,
\end{equation}
and consider the following class $\Xi$ of Incentive Compatible contracts, see Proposition~\ref{IC} below,
$$
\Xi=\{ \xi=Y_T^{(y,\beta)} \in \mathcal{F}_T^X, \text{ where }y \ge y_0, \beta=(\beta_t)_{t\leq T} \in \CA \text{ and } Y_T^{(y,\beta)}=y+\int_{0}^T f^\ast(\beta_s) ds +\int_0^T \beta_s dB_s \}.
$$
We have:
\begin{theorem}\label{Main}
The optimal contract $\xi^*$ that maximizes the principal problem \eqref{eq:Principalpb} is linear in end-of-period profits $X_T$ and is given by
\begin{align}\label{eq:Maincontract}
\xi^*=y_0  -\frac{\gamma_P+1/\kappa}{\gamma_A+\gamma_P+1/\kappa} g_0(T)+ \frac{\kappa\gamma_A-1}{2\kappa} \int_0^T (\beta^*_s)^2\,ds+\frac{\gamma_P+1/\kappa}{\gamma_A+\gamma_P+1/\kappa}X_T,
\end{align}
and  the optimal level of recommended effort $a^*$ that maximizes the agent's problem \eqref{agentutility} is deterministic and given by $a^*=\frac{\beta^*}{\kappa}$ with
\begin{align}\label{eq:Maineffort}
\beta^*_t=\frac{\gamma_P+1/\kappa}{\gamma_A+\gamma_P+1/\kappa}K(T,t), \quad t \leq T.
\end{align}
\end{theorem}

\begin{proof}
		See Section~\ref{sectionP}. 
\end{proof}

Similarly to HM, the optimal compensation is made up of a deterministic base salary $y_0 -\frac{\gamma_P+1/\kappa}{\gamma_A+\gamma_P+1/\kappa} g_0(T)+ \frac{\kappa\gamma_A+1}{2\kappa} \int_0^T (\beta^*_s)^2\,ds$ and a random compensation to foster incentives $\frac{\gamma_P+1/\kappa}{\gamma_A+\gamma_P+1/\kappa}X_T$.
 One of the striking results is, when agents have CARA preferences,  the incentive part of the optimal contract, through the performance-based bonus coefficient $\frac{\gamma_P+1/\kappa}{\gamma_A+\gamma_P+1/\kappa}$, is common to all one-dimensional Volterra Gaussian models and thus independent of the output dynamics, even though they have very different statistical properties. Only the base salary is industry-specific depending on the output dynamic through the Volterra kernel $K$. 
The optimal effort level is deterministic and firm-specific and can, depending on the choice of the Volterra kernel, exhibit interesting behaviors. 
For instance, for the mean-reverting dynamics, $i.e.~K(t,s)=e^{-\lambda (t-s) } \ind_{s<t}$, the optimal effort is increasing if the mean-reverting intensity  $\lambda$ is positive.  The closer one gets to contract maturity, the more work the agent has to do. The intuition is that the optimal effort should compensate for the natural tendency of the process to revert to its long-term average. The closer the contract is to maturity, the greater the effort should be to allow $X$ to deviate from its long-term average and thus allow the principal to benefit from a greater profit.
When the mean-reverting intensity is negative, the effort must be greater at the beginning of the contract in order to give the necessary impetus to the process to diverge towards large positive values. Once this momentum is established, it is less effective to ask the agent to work.\\

The following two sections are dedicated to proving Theorem~\ref{Main}. Section~\ref{agentpb} solves the agent problem, while section~\ref{sectionP} solves the principal problem. An extension of Theorem~\ref{Main} to the multi-dimensional set-up is considered in Section~\ref{S:multid}.

\section{The  one-dimensional agent problem}\label{agentpb}
This section aims at  completely  solving the problem of the Agent in \eqref{agentutility}. The ideas developed here are not new, they rely on the martingale approach to stochastic control already used in \citet{schattler1993first} which we adapt to develop the first-order approach to principal-agent problems in a continuous-time Gaussian setting with exponential utilities. We will  show that the class $\Xi$ of contracts are incentive compatible contracts and design the optimal response of the agent for a given contract in $\Xi$. Our construction relies on the following martingale optimality principle that brings a clear intuition of the stochastic maximum principle used in the context of dynamic contracting by \citet{williams2011ECMA}.

\subsection{The Martingale optimality principle}
The Martingale optimality principle must be seen as a sufficient condition for a contract to be incentive-compatible. The following lemma, which is due to \citet{hu2005utility}  and  proved in Appendix \ref{MOP} for completeness, states this principle.

\begin{lem} \label{martingaleoptimality} Given a contract $\xi$, suppose the existence of a family of stochastic processes $R^a(\xi):=(R_t^a)_{t\leq T}$ indexed by $a \in \CA $ such that the following four assertions hold
\begin{itemize}
\item[i)] $R_T^a = U_A\left(\xi - \int_0^T k(a_s) ds\right), \quad \forall a \in \CA$,
\item[ii)] $R_{\cdot}^a$ is a $((\cF_t)_{t\in [0,T]},\PP^a)$-supermartingale for every $a$ in $\CA$,
\item[iii)] $R_0^a$ is independent of $a$,
\item[iv)] there exists $a^\ast$ in $\mathcal{A}$, such that $R^{a^\ast}$ is a $(\mathcal F_t)_{t\in [0,T]}$-martingale.
\end{itemize}
Then, $\xi$ is incentive compatible for the Agent problem \eqref{agentutility} and $a^\ast$ is the agent best reply.
\end{lem}

In the dynamic agency literature, the process $(R_t^{a^\ast})_{t\leq T}$ describes the Agent's expected utility given the contract $\xi$.
A contract $\xi$ is thus incentive compatible if we are able to build such a family $R^a(\xi)$. This will be done in the next section.

\subsection{Enlarging the class of Incentive Compatible Contracts}
 In accordance with the result of \citet{schattler1993first}, we expect that the contracts belonging to $\Xi$ are incentive compatible. It is at this point that a difficulty arises in our setting compared to the Brownian model of \citet{holmstrom1987aggregation} and more generally to the standard literature where the information sets of the two players coincide in the absence of moral hazard. Because the principal has a coarser information (recall that the paths of $B$ are not always observable by the principal), she cannot in general implement the process $(Y_t^{y,\beta})_{t\leq T}$ given by 
		\begin{align}\label{eq:Yproc}
Y_t^{y,\beta} = y + \int_0^t f^*(\beta_s)ds + \int_0^t \beta_s dB_s,		
		\end{align}
		  for  $y\geq y_0$ and $\beta \in \CA$, because $Y_t^{y,\beta}$ fails to be $\mathcal F^X_t$-measurable. In other words, the contracts in $\Xi$, that are the most natural to be incentive compatible are a priori inaccessible, unless we are able to characterize the
 controls $\beta \in \CA$ that induce $Y_T^{(y,\beta)} \in \mathcal{F}_T^X$.  Putting aside for a while this problem of information asymmetry between the two players, we consider a larger game where the principal is supposed to have the same information as the agent. We will forget for a while the constraint $\xi \in \mathcal F^X_T$ and introduce the enlarged set of contracts
 $$
 \hat \Xi=\{ \xi=Y_T^{(y,\beta)} \text{ where } y \ge y_0, \, \beta \in \CA,\quad Y_T^{(y,\beta)}=y+\int_{0}^T f^\ast(\beta_s) ds +\int_0^T \beta_s dB_s\}
 $$
 and naturally extend Definition~\ref{def:ic} of incentive-compatibility for $\mathcal F^B_T$-measurable contracts. We have the following result that we prove for sake of completeness using the Martingale optimality principle.

\begin{pro}\label{IC}
	Let $\hat \xi \in \hat \Xi$ be of the form 
	$$\hat \xi = y + \int_0^T f^*(\beta_s)ds + \int_0^T \beta_sdB_s,$$ with $y\geq y_0$ and $\beta \in \mathcal A$. Then,    $\hat \xi$ is incentive compatible for the Agent problem in \eqref{agentutility} and satisfies the participation constraint. Furthermore, the agent best reply is given by the effort $a^{*}(\hat \xi)= \frac{\beta}{\kappa}$  and the utility of the agent at $0$ is given by $V_0^A(\hat \xi)=-\exp(-\gamma_A y)$. 
\end{pro}
 \begin{proof}
Fix $\hat \xi \in \hat\Xi$. Let $y\geq y_0$ and $\beta \in \mathcal A$ such that $\hat \xi =y+\int_{0}^T f^\ast(\beta_s) ds +\int_0^T \beta_s dB_s$ and define the process $Y^{(y,\beta)}$ by  \eqref{eq:Yproc}. 
 For an admissible effort policy $a=(a_t)_{t\leq T} \in \CA$, we define $R^a$ as 
$$ R^{a}_t:=-\exp\left(-\gamma_A \left(Y_t^{(y,\beta)} - \int_0^t k(a_s) ds \right)\right), \quad t\in [0,T].$$
We will show that the family $(R_t^a)_{t\leq T}$ satisfies condition i)--iv) of the Martingale optimality principle of Lemma~\ref{martingaleoptimality}. 	Observe that $Y_T^{(y,\beta)}= \hat\xi$ so that Lemma~\ref{martingaleoptimality}-i) is satisfied. Also, $R^a_0 = - \exp(-\gamma_A y)$ is independent of $a$ as needed in Lemma~\ref{martingaleoptimality}-iii). Furthermore,    
recalling that $B^a=B-\int_0^{\cdot} a_s\,ds$, we note that
\begin{align*}
Y_t^{(y,\beta)} - \int_0^t k(a_s) ds  &= y + \int_{0}^t \left(f^\ast(\beta_s)- k(a_s)\right) ds +\int_0^t \beta_s dB_s \\
&= y + \int_{0}^t \left(f^\ast(\beta_s)- k(a_s) + a_s \beta_s \right) ds +\int_0^t \beta_s dB^a_s.
\end{align*} 
Using the definition of $f^{\ast}$ in \eqref{BSDEdrift}, a completion of the squares  in $a_s$ yields the expression
$$ f^\ast(\beta_s)- k(a_s) + a_s \beta_s = -\frac{\kappa}{2}\left(a_s - a^\ast_s\right)^2 + \frac{\gamma_A}{2} \beta^2_s $$ 
with $a^\ast_s:=\frac{\beta_s}{\kappa}$
 so that combining the above leads to 
\begin{align*}
R_t^a =  - \exp\left( -\gamma_A y\right) \exp\left(\frac{\gamma_A\kappa}{2}\left(a_s -a^\ast_s\right)^2 \right)\exp\left(  -\frac{\gamma_A^2}{2} \int_0^t \beta_s^2 ds  - \gamma_A\int_0^t \beta_s dB^a_s\right)
\end{align*} 
 It remains to argue that the process $M^a:= \exp\left(-\gamma_A \int_0^{\cdot} \beta_sdB^a_s-\frac{\gamma_A^2}{2}\int_0^{\cdot}\beta_s^2\, ds\right)$ is a martingale under $\P^a$. Indeed, if this is the case, then, since $-\exp(\frac{\gamma_A\kappa}{2}(a_s -a^\ast_s)^2 \leq -1$,  
 $$ \mathbb E^a[R^a_T]  \leq -\exp(-\gamma_A y) \mathbb E^a[M_T^a] = -\exp(-\gamma_A y) = R^a_0,  $$ 
 which shows that $R^a$ is a $\PP^a$-supermartingale for each $a\in \mathcal A$, which corresponds to condition Lemma~\ref{martingaleoptimality}-ii). Furthermore, for $a=a^{\ast}$, we have that $R^{a^*}$ is a $\PP^{a^{\ast}}$-martingale which gives Lemma~\ref{martingaleoptimality}-iv). 
Obtaining  that  $M^a$ is a martingale under $\PP^a$ is equivalent to proving that 
$$
M_t=\exp\left( \int_0^t {a_s}\,dB_s-\frac{1}{2}\int_0^t a_s^2\,ds\right) \exp \left(-\gamma_A \int_0^t \beta_s \,dB^a_s-\frac{\gamma_A^2}{2}\beta_s^2\, ds\right)
$$
is a martingale under $\P^0$. But, observe that
$$
M_t=\exp\left( \int_0^t {(a_s-\gamma_A\beta_s)}\,dB_s-\frac{1}{2}\int_0^t (a_s-\gamma_A\beta_s)^2\,ds\right) 
$$
which is a martingale for $(a_t)_t$ and $(\beta_t)_t$ in $\mathcal A$. An application of Lemma \ref{martingaleoptimality} shows that $\hat\xi$ is incentive compatible such that the agent best reply is given by the effort $a^*_s(\hat \xi) = \frac{\beta_s}{\kappa}$. Finally, since $y\geq y_0$, the identity 
$$ V_0^A(\xi) =\mathbb E^{a^*}[R^{a^*}] = -\exp(- \gamma_A y) \geq -\exp(-\gamma_A y_0)=R_0, $$
shows that $\hat\xi$ satisfies the participation constraint  by giving  the required utility of the agent at $0$, which concludes the proof. 
\end{proof}
Notably, when the principal offers a contract parametrized by the pair $(y,\beta)$, the agent best reply is $\frac{\beta}{\kappa}$ and thus independent of $y$. This is due to the no wealth effect of CARA  preferences. The agent utility is $-\exp({-\gamma_A y})$ and thus independent of $\beta$.  This is due to the agent's full commitment allowing the principal to choose the best incentive contract that binds the participation constraint. To sum up, restricting our attention to contracts in $\hat \Xi$ transforms the puzzling principal's problem to a stochastic Volterra control problem, namely\footnote{ To alleviate notations, we will denote hereafter $\P^\beta$, the probability corresponding to the agent's effort choice $a=\frac{\beta}{\kappa}$.} 
\begin{equation}\label{secondbest}
V_{SB}=\sup_{y\ge y_0,\beta \in \CA} \E^\beta\left[ U_P\left(X_T-Y_T^{y,\beta}\right) \right]=\sup_{\hat \Xi} \E\left[ U_P\left(X_T-\hat\xi\right) \right].
\end{equation}
where $(Y_t^{y,\beta})_{t\leq T}$  is given by \eqref{eq:Yproc}.

The principal problem \eqref{secondbest} corresponds to the enlarged stochastic control problem where the principal would have access to the information generated by the Brownian motion. 
Clearly, the principal value \eqref{eq:Principalpb} satisfies $V_0^P \le V_{SB}$ because of the inclusion $\Xi \subset \hat \Xi$.
In the one-dimensional Brownian model,  \citet{holmstrom1987aggregation}  show that the two values coincide because the sets of information $\FF^B$ and $\FF^X$ are identical  ($\Xi=\hat \Xi$) and thus there is no need to introduce the { enlarged control problem.
Our contribution will be to show that the two values always coincide for one-dimensional Gaussian Volterra models, even if $\FF^X$ is strictly included in $\FF^B$. This is the object of the next section.

\section{The one-dimensional principal Gaussian problem}\label{sectionP}

This section is devoted to the explicit resolution of the  principal problem \eqref{secondbest} and to the proof of Theorem \ref{Main}. Contrary to the standard literature, the problem \eqref{secondbest}  is not a Markovian stochastic control problem because the process $X_t$ is not necessarily Markov. More precisely, it corresponds to a stochastic Volterra control problem with the following controlled processes 
\begin{align*}
X_t &= g_0(t) + \frac{1}{\kappa} \int_0^t K(t,s)\beta_s ds + \int_0^t K(t,s)dB_s^{\beta},\\
Y^{y,\beta}_t &= y + \frac{\kappa\gamma_A + 1}{2\kappa} \int_0^t \beta_s^2 ds + \int_0^t \beta_s dB^{\beta}_s.
\end{align*}
We will show that the optimal second-best contract exists and is furthermore  $\mathcal F^X_T$-measurable. As a consequence, the second-best principal value $ V_{SB}$ will coincide with the principal value $V_0^P$. In other words, our main message is that there is no gain to the principal in acquiring more information than that generated by the observed output process $X$ in one-dimensional Gaussian Volterra models, regardless of the definition of the kernel $K$. For instance, in the example of discrete observations of Brownian motion,  i.e.~$K(t,s)=f(t)\ind_{s \le t}$ with $f$ as in \eqref{eq:fdis}, there is no gain to the principal in increasing the frequency of the discrete observations of the Brownian output.\\

For  $y \geq y_0$ and  a control policy $\beta \in \CA$, we define 
$J(y,\beta)= \E^{\beta} \left[ \exp\left( - \gamma_P\left(X_T - Y_T^ {y, \beta} \right) \right) \right],$
in order to write the second-best principal problem
\begin{equation} \label{relaxed}
 V_{SB} = \inf_{y\geq y_0}V_{SB}(y), \, \text{ with }V_{SB}(y) =\inf_{\beta \in \CA} J(y,\beta). 
 \end{equation}
The rest of the section is dedicated to the proof of Theorem \ref{Main} that characterizes the optimal control for the principal problem \eqref{secondbest}.
The idea of the proof of Theorem \ref{Main} is to apply again the martingale optimality principle. To do this, we need to introduce a good family of processes indexed by $\beta$. Inspired by the agent problem, one possibility would be to consider the following family
$$
\exp\left(-\gamma_P \left( X_t - Y_t^{ y, \beta} \right) \right).
$$
Unfortunately, it may be impossible to apply It\^o's formula since the process $X$ may not be a semi-martingale, as in the fractional Gaussian processes case. To get around this problem, we introduce a new state variable that can be interpreted as a forward price which is a semi-martingale that coincides with $X$ at date $T$.
Let us define the {\it effort-corrected forward output} process by 
$$
g_t^{\beta}(T) = \E^{\beta}\left[ X_T -\frac{1}{\kappa}\int_t^T K(T,u) \beta_u du \mid \mathcal F_t\right].
$$
Using the output dynamics \eqref{outputeffort} with effort $\beta \in \CA$, we have
$$
g^\beta_t(T)= g_0(T) + \frac{1}{\kappa} \int_0^t K(T,u) \beta_u du + \int_0^t K(T,u) dB_u^{\beta}.
$$
Then, we observe that the process $(g^{\beta}_t(T))_{t \le T}$ is a semi-martingale on  $[0,T)$ with dynamics 
\begin{align}\label{eq:dynamicsg}
dg^{\beta}_t(T) = \frac{1}{\kappa}K(T,t) \beta_t dt + K(T,t) dB_t^{\beta}
\end{align}
and terminal value $g^\beta_T(T)=X_T$.
To apply the martingale optimality principle, we will consider the family of processes
\begin{align*} M_t^{\beta} = \exp\left(-\gamma_P \left( g^{\beta}_t(T) - Y_t^{ y, \beta} \right)  + \phi_t   \right),
\end{align*} 
where $\phi$ is the deterministic function given by 
$$ \phi_t  =  \frac{\gamma_P} 2 \left( \gamma_P^2 -  \frac{(\gamma_P + 1/\kappa)^2}{(\gamma_A  +\gamma_P+ 1/\kappa)}   \right) \int_t^T K(T,s)^2ds . $$
Lemma \ref{L:M}, which is proved in Appendix~\ref{MOP}, provides the dynamics of $M^{\beta}$ that plays a key role in the determination of the optimal contract.
\begin{lemma}\label{L:M}
	For each $\beta \in \CA$, we have 
	\begin{align}\label{eq:Mdynamics}
	\frac{dM_t^{\beta}}{M_t^{\beta}} = \frac{\gamma_P} 2 (\gamma_A  +\gamma_P+1/\kappa) (\beta_t - \beta_t^*)^2 dt +  \left(\gamma_P \beta_t -\gamma_P K(T,t) \right)  dB_t^{\beta}, \quad \mathbb P^{\beta}-a.s.  
	\end{align}
	with $\beta^*$ given by \eqref{eq:Maineffort}.
\end{lemma}

\begin{proof}
		See Appendix~\ref{MOP}. 
\end{proof}

We can now complete the proof of Theorem~\ref{Main}.
\begin{proof}[Proof of Theorem~\ref{Main}] 
	$ \bullet$ \textit{The Principal's problem} is solved by an application of the martingale optimality principle on the process $M^{\beta}$. 	Fix $\beta \in \mathcal A$ and $y\geq y_0$.  We show that the family $M^{\beta}$ satisfies the four assertions of the martingale optimality principle in Lemma~\ref{martingaleoptimality}. We have,
\begin{itemize}
\item[i)] For all $\beta \in \CA$, it follows from \eqref{eq:Mdynamics} that $M^{\beta}$ is a $\P^{\beta}$-sub-martingale and thus
$$  M_0^{\beta} \leq \E^{\beta}\left[M_T^{\beta}\right] = \E^{\beta} \left[ \exp\left( -\gamma_P \left( X_T - Y_T^{y,\beta} \right) \right) \right]=J(y, \beta),$$
where we used that $g^{\beta}_T(T)= X_T$.
\item[ii)] Observe that $M_0^\beta=\exp\left(-\gamma_P \left( g_0(T) - y \right)  + \phi_0   \right)=M_0$ is independent of $\beta$.
\item[iii)]  Finally, for $\beta^*$ given by \eqref{eq:Maineffort}, $M^{\beta^*}$ is a $\P^{\beta^*}$-martingale by \eqref{eq:Mdynamics} and thus, we have 
$$ J(y,\beta^*)  =  \E^{\beta}\left[M_T^{\beta^*}\right] = M_0^{\beta^*} =  M_0 \leq J( y,\beta), \quad \beta \in \CA, $$ 
which shows that $\beta^*$ is the optimal control for the second-best principal problem and the principal value is given by $V_{SB}(y)=M_0$. 
\end{itemize}
Optimizing on $y\geq y_0$, yields $V_{SB}=\exp\left(-\gamma_P \left( g_0(T) - y_0 \right)  + \phi_0   \right)$.
Furthermore, since $\beta^*_t$ is proportional to the Volterra Kernel $K(T,t)$, it is straightforward to obtain the linear form of the contract $\xi^*=Y_T^{y_0,\beta^*}$  as in \eqref{eq:Maincontract}. In particular, $\xi^*$  is $\FF^X_T$-measurable as an affine function of $X_T$. Therefore, the optimal control for the enlarged principal problem \eqref{secondbest} induces an optimal contract in $\Xi$, so that $V_{SB}=V_0^P$.\\
$\bullet$\textit{The Agent's problem.} An application of Proposition~\ref{IC} yields that the optimal level of recommended effort $a^*$ that maximizes the agent's problem \eqref{agentutility}  is given by $a^*=\frac{\beta^*}{\kappa}$. 
\end{proof}

To sum up, this study shows that the transition from Brownian to Volterra models preserves the optimality of linear in end-of-period profit contracts. Moreover, in one-dimensional models, the principal does not suffer from having a 
coarser information than the agent about the dynamics of the production process. Aggregating production over time is sufficient for optimal compensation in Volterra Gaussian environments and there is no need to use all the available information - the brownian path in our setting- to design an optimal contract. The next section deals with the robustness of this result in the multi-dimensional set-up.

\section{The multi-dimensional model} \label{S:multid}
So far, we have assumed that the shocks are modeled by a one-dimensional Brownian motion. In this section,
we present a tractable class of multitask principal-agent problems, such as the one faced by a firm with a manager that supervises several projects. This model amounts to study the principal-agent problem in the case where the shocks are modeled by a standard Brownian motion of dimension $d$, that we also denote by $(B_t)_{t\le T}$.  As in Holmstrom and Milgrom,  the $i$-th component $B^i_t$ of $B_t$ is interpreted as the  outcome of the $i$-th account of the firm. We model the aggregate output or profit of the firm as follows
$$
X_t=g_0(t)+\int_0^t <K(t,s),dB_s>,
$$
where $<\cdot,\cdot>$ denotes the canonical inner product in $\R^d$, $g_0$ is a deterministic function and  $K : [0,T]^2 \to \R^d$ is a measurable Volterra Kernel, i.e.  $K(t,s)=0$ for $s \geq  t$ such that 
\begin{align}
\sup_{t\leq T} \int_0^T ||K(t,s)||^2 ds < \infty. 
\end{align}
For instance,  the case $K(t,s)=\sigma\ind_{s \le t},\,\sigma \in \R^d$ corresponds to the Brownian model in  \citet[Section 4]{holmstrom1987aggregation}. 
We can also consider a mining company that exploits two different types of minerals in two different mines. Each component of $X$ represents the revenue of a mine. This may correspond to the choice $K(t,s)=(e^{-\lambda_1(t-s)},e^{-\lambda_2(t-s)})$ where each separate outcomes follows a mean-reverting process with two different mean-reverting intensity $\lambda_i,\,i=1,2$.\\
Even more importantly than in the one-dimensional case, the filtration generated by the output $\FF^X$ is strictly included in the filtration generated by the multi-dimensional Brownian motion $\FF^B$ and thus, the principal has always a coarsest information than the agent even when the latter does not exert any hidden effort. This observation is central to Holmstrom and Milgrom's distinction between the optimal contracts that are linear  in profits or in accounts and more generally to understand when it is useless to use all the information generated by the Brownian motion. It is useful to recall here that
we focus in this paper on contracts that are $\mathcal F_T^X$ measurable.
\\
In a similar way to the one-dimensional case, we assume that the agent can exert a continuous vector of effort $(a_t)_{t\leq T} \in \R^d$, $a_t^i$ being the effort made by the manager to improve the account $i$, that modifies the probability distribution of $X$ as follows:
$$
X_t=g_0(t)+\int_0^t <K(t,s),dB_s +a_s\,ds>,
$$
Similarly to the one-dimensional case, we say that an agent's action $a=(a_t)_t$ is admissible if $a=(a_t)_t$ is $\FF$-progressively measurable and such that there exists $A>0$ such that 
$$
\int_0^T||a_s||^2\,ds < A, \quad \PP^0-a.s. 
$$ 
Still denoting by $\CA$ the set of admissible actions, we define for any $a \in \CA$ a family of equivalent probability measures $\PP_a$ by
$$
\frac{d\PP_a}{d\PP_0}=\exp\left( \int_0^T <a_s,dB_s>-\frac{1}{2}\int_0^T ||a_s||^2\,ds\right).
$$
Under $\PP_a$, the process $B^a=B-\int_0^{\cdot} a_s\,ds$ is a $\FF-$Brownian motion and the output dynamics is
\begin{equation}\label{outputeffortdimd}
X_t=g_0(t)+\int_0^t <K(t,s), dB_s^a+a_s\,ds>.
\end{equation}
We also assume that the agent incurs an instantaneous cost $k(a)$ where $k$ is a convex function on $\R^d$ with $k(0_{\R^d})=0$.  When the kernel is a constant vector, Holmstrom and Milgrom have considered the case where the effort cost function is $k(a)=g(\sum_{i=1}^d a_i)$ with $g$ strictly convex and have showed that the optimal compensation is linear in profit in that case. Nevertheless, this specification does not allow us to determine the optimal effort that the agent must make in each of his tasks. In this section, we will rather consider a quadratic effort cost function 
\begin{align*}%\label{eq:quadeffort}
k(a)=\frac{1}{2}<a,\Gamma a>,
\end{align*}
 where $\Gamma$ is a symmetric positive-definite matrix.

When $\Gamma$ is proportional to the identity matrix,  i.e.~$\Gamma = \kappa I_d$ for some $\kappa >0$, we say that the effort cost function is radial, because, in this case, the effort cost function is proportional to the norm of the vector $a$. 
A radial cost is to assume that the effort costs are not specific to the different tasks that define the accounts. 

In the sequel, we will highlight the interplay between the choice of the matrix $\Gamma$ and the optimality of linear contracts. In a nutshell, our main results of this section in the multi-dimensional framework can be summarized as follows:
\begin{itemize}
	\item If $\Gamma$ is proportional to the identity matrix, then the optimal contract  $\xi^*$ is linear in the end-of-period profit $X_T$.  As in the one-dimensional model, the principal does not have to worry about her lack of information to sign an optimal contract.
	\item For more general matrices $\Gamma$, the optimal contract 
$\xi^*$ is no longer linear  in the end-of-period profit $X_T$. More importantly,  $\xi^*$ is not necessarily $\mathcal F_T^X$ measurable, meaning that the less-informed principal cannot implement/sign the contract $\xi^*$. In this situation, we quantify the gap between such contract $\xi^*$ and the best linear contract that can be implemented by the principal.  The gap can be interpreted as the value of information. 
\end{itemize}

\subsection{The agent's problem}
From a methodological viewpoint, there is no hurdle to adapt the techniques developed in Section \ref{agentpb}. As a consequence, we will roughly repeat the approach detailed in  Section \ref{agentpb} to apply the martingale optimality principle and consider a class of incentive-compatible contracts. To do this,  we assume for a while that the principal has access to the Brownian filtration generated by $(B_t)_t$ and can implement the controlled process
\begin{align}\label{eq:Y-beta-d}
Y_t^{ y,\beta}=y+\int_0^t f^*(\beta_s)dt+\int_0^t<\beta_s, d B_s>,
\end{align}
with $y \ge y_0$ and $\beta \in \CA$,  and for $z \in \R^d$,
\begin{align*}
f^*(z)=\frac{\gamma_A}{2} ||z||^2+\inf_a \left( \frac{1}{2}<a,\Gamma a> -<a,z> \right)  = \frac{1}{2} <z , \left(\gamma_A I_d - \Gamma^{-1}\right) z>
\end{align*}
to offer the wage $\xi=Y_T^{y,\beta}$ which is $\mathcal F^B_T$ measurable.  For a given contract $\xi$ defined by $(y,\beta) \in [y_0,\infty)\times \mathcal A$, the agent has to determine his best response $a^*(\xi)$. To apply the  martingale optimality principle, we introduce the family of processes indexed by $a$ given by
$$
R_t^a=-\exp\left( -\gamma_A \left(Y_t^{ y, \beta}-\int_0^t \frac{1}{2}<a_s,\Gamma a_s>\,ds\right)\right)
$$
The first-order condition gives the agent's best effort, $a^*(\xi)_t=\Gamma^{-1}\beta_t$. 
Furthermore, the agent utility  at time $0$ is given by $V_0^A(\xi) = R_0^a = -\exp(-\gamma_A y).$ 
We collect the result in the following proposition, which is the analogue of Proposition~\ref{IC} in the multi-dimensional framework. 
\begin{pro}\label{IC-d}  Let $\hat \xi$ be a contract of the form 
		$$\hat \xi = y + \int_0^T f^*(\beta_s)ds + \int_0^T <\beta_s, dB_s>,$$ with $y\geq y_0$ and $\beta \in \mathcal A$. Then,    $\hat \xi$ is incentive compatible for the Agent problem \eqref{agentutility} and satisfies the participation constraint. In particular, the agent best reply is given by the effort $a^{*}(\hat \xi)=\Gamma^{-1} {\beta}$ and the  agent utility  at time $0$ is given by $V_0^A(\hat \xi) = -\exp(-\gamma_A y).$
		\end{pro}

We  can again re-write the Principal's problem as Volterra stochastic optimal control problem on the process
$$
Y_T^{ y,\beta}=y+\frac{1}{2}\int_0^T <\beta_t,\left(\gamma_A Id+\Gamma^{-1}\right)\beta_t>\,dt+\int_0^T<\beta_t, d B^\beta_t>,
$$
where the stochastic process $B^\beta=(B^\beta_t)_t$ is a d-dimensional Brownian motion under the probability measure indexed by $a^*_t=\Gamma^{-1}\beta_t$ that we will denote hereafter $\P^*$.\\
Then, without asymmetry of information,  the enlarged principal problem is given by
 $$ V_{SB} = \sup_{y\geq y_0} V_{SB}(y), $$
	with 
\begin{align}\label{eq:Vsb_d}
V_{SB}(y)=\sup_{\beta \in \CA}\E^* \left[ U_P\left( X_T-Y_T^{y,\beta} \right) \right]
\end{align}
and is an upper bound for the principal problem \eqref{eq:Principalpb} with the constraint $\xi \in \mathcal F^X_T$.

\subsection{The Enlarged Principal problem}
The idea is to mimic the methodology developed in details in the one-dimensional case. For this purpose, we reintroduce the effort-corrected forward output
$$
g_t^\beta(T)=\E\left[ X_T-\int_t^T <K(T,s), \Gamma^{-1} \beta_s>\,ds | \mathcal{F}_t  \right]
$$
and apply the martingale optimality principle to the process
\begin{align*} M_t^{\beta} = \exp\left(-\gamma_P \left( g^{\beta}_t(T) - Y_t^{y,\beta} \right)  + \phi_t   \right),
\end{align*} 
where $\phi$ is the deterministic function to determine.\\

  The following theorem gives the  optimal contract in a multi-dimensional setting for the enlarged principal problem.
	
	\begin{theorem}\label{Main-d}
	Let $\Gamma$ be symmetric positive-definite. 
	 The optimal level of  effort  $\Gamma^{-1} \beta^*$ that maximizes the enlarged principal's problem \eqref{eq:Vsb_d}  is deterministic and $\beta^*$ is given by
	\begin{align}\label{eq:betavsb-d}
	\beta^*_t=\left(\left(\gamma_A+\gamma_P\right) I_d+\Gamma^{-1}\right)^{-1}
\left(\gamma_P I_d+\Gamma^{-1}\right)K(T,t), \quad t \leq T.
	\end{align}
	The utility of the principal at time $0$ is given by 
	$$ V_{SB}  = V_{SB}(y_0) $$
	with 
	\begin{align}\label{eq:Vsb-explicit}
	 V_{SB}(y)= -\exp\left(-\gamma_P(g_0(T)-y) + \phi_0\right)
	\end{align}
	and
	\begin{align*}
	\phi_0 = \frac{\gamma_P}{2} \langle K_T, \left(\gamma_P I_d - \left(\gamma_P I_d + \Gamma^{-1}\right) \left(\left(\gamma_A + \gamma_P\right) I_d + \Gamma^{-1}\right)^{-1} \left(\gamma_P I_d + \Gamma^{-1}\right)\right) K_T\rangle_{L^2},
		\end{align*}
		where $K_T(s):=K(T,s)$ and $\langle f,g \rangle_{L^2}:=\int_0^T <f(s),g(s)>ds$. 
			The optimal contract $\xi^*$ that maximizes the principal problem  is  given by 
\begin{align}\label{eq:optimalcontract-d}
		\xi^*=y_0+  \int_0^T f^*(\beta^*_s) ds + \int_0^T <\beta^*_s,dB_s>. 
\end{align}
\end{theorem}

\begin{proof}
	See Appendix~\ref{MOP}.
\end{proof}

We now make two important observations.
\begin{rem}\label{R:enlarged}
\begin{itemize}
	\item 
	We note that in general, the optimal contract \eqref{eq:optimalcontract-d} is not linear in $X_T$, indeed the term $\int_0^T <\beta_s^*, dB_s\rangle $ with $\beta^*$ given by \eqref{eq:betavsb-d} cannot be expressed in terms of the integral $\int_0^T <K(T,s),dB_s>$. 
	\item 	More importantly, $\xi^*$ is measurable with respect to $\mathcal F^B_T$ but not necessarily with respect to the smaller filtration $\mathcal F^X_T$, which means that such contract cannot be implemented by the  less-informed principal.  
\end{itemize}
\end{rem}

The following corollary shows that if the cost   is radial then the optimal contract is linear in $X_T$ and it can therefore  be implemented by the principal.  In Section \ref{S:valueofinfo} below, we study a class of optimal linear implementable contracts for the principal in case the cost is not radial.

\begin{corollary}\label{radial}
	Assume that the effort cost function is radial, i.e.~$\Gamma= \kappa Id$ for some $\kappa>0$. Then, the optimal level of effort that maximizes the enlarged principal's problem \eqref{eq:Vsb_d}  is deterministic and given by	
	\begin{align}\label{eq:betastard}
	\beta^*_t=\frac{\gamma_P+{1}/{\kappa}}{\gamma_A+\gamma_P+{1}/{\kappa}}K(T,t).
	\end{align}
In particular,  the optimal contract $\xi^*$ is linear in profits and given by
\begin{align*}
	\xi^*=y_0-\frac{\gamma_P+1/\kappa}{\gamma_A+\gamma_P+1/\kappa} g_0(T)+ \frac{\gamma_A + {1}/{\kappa}}{2} \int_0^T <\beta^*_s, \beta^*_s> ds + \frac{\gamma_P+{1}/{\kappa}}{\gamma_A+\gamma_P+{1}/{\kappa}} X_T.
\end{align*}
Furthermore, $V_{SB}=V_0^P$.
\end{corollary}

\begin{proof}
	The expression for $\beta^*$ and $\xi^*$ follow directly from Theorem~\ref{Main-d}.  In particular, $\xi^*$  is $\mathcal F^X_T$-measurable as an affine function of $X_T$. Therefore, the optimal control for the enlarged principal problem \eqref{eq:Vsb_d}  induces a contract which is implementable by the principal, so that $V_{SB}=V_0^P$
\end{proof}

The message of  Corollary~\ref{radial} is very simple. When an agent has to allocate his time on several tasks and when his effort cost function is not specific to tasks and thus measured by the norm of the vector $a$, it is not necessary for the principal to scrutinize the revenues of each activity. It is sufficient to sign a linear contract in the final value of the aggregate profits to give the optimal incentives without regard to the optimal level of information. When the effort cost function is radial, the principal need not observe the paths of individual accounts to offer an optimal compensation that is linear in profits.  With this specification, we can disentangle the optimal efforts to allocate to the different tasks because the ith component to the optimal effort is proportional to the ith component of the kernel. In order to illustrate how an agent should optimally allocate his time to the different tasks he has to perform, let us consider the following toy example. A salesperson must visit two clients in two different geographical areas. We assume that the first geographical area generates Brownian outcomes and the second area generates  mean-reverting outcomes.
	The firm's aggregate output is given by
	$$
	X_t=B_t^1+\int_0^t e^{-\lambda (t-s)}\,dB_s^2, \text{ with }\lambda >0. 
	$$
	While the share of the output that goes to the agent is independent of the parameter $\lambda$, the saleperson must differentiate his customers visit. The first customer's visit must be on a constant basis and the second customer's visit must be accelerated as the maturity of the contract approaches.\\ 

\subsection{ The subclass of linear contracts and the value of information}\label{S:valueofinfo}
In this section, we  no longer assume that the cost is radial, we study the subclass of linear contracts,  and we quantify the incurred loss on the utility of the principal.
Beyond their simplicity,  the main advantage of linear contracts is that they are $\mathcal F_T^X$ measurable and can therefore be implemented by the  less-informed principal. 
We will define the value of information as the premium the principal would have to pay to access the agent's information and implement the optimal contract. 
We will consider contracts as in \eqref{eq:Y-beta-d} with $y\geq y_0$ but only for  controls $\beta$ in the form 
$$ \beta_t = b K(T, t), \quad  \mbox{for some } b\in \mathbb R. $$
Note that in this case, Proposition~\ref{IC-d} still apply to ensure that such contracts are Incentive Compatible  and the agent best reply is still $\Gamma^{-1}\beta$. Furthermore, for any $b\in \R$ the contract $Y_T^{y,b}$ is by construction linear in $X_T$ and given by 
\begin{align*}
Y_T^{y,b} = y + \int_0^T f^*(bK(T,s))ds + b (X_T-g_0(T)),
\end{align*}
so that $b$ is the share of the output that goes to the agent.
 In this case, the principal will optimize on $(y,b)$ to find the optimal linear contract
 \begin{align*}
 V_{{lin}} = \sup_{y\geq y_0}  V_{{lin}}(y), 
 \end{align*}
 with 
 \begin{align}\label{eq:Vlin_d}
 V_{{lin}}(y)=\sup_{b \in \mathbb R}\E^* \left[ U_P\left( X_T-Y_T^{y,b} \right) \right]
 \end{align} 
 
 A direct computation and optimization of the expectation leads to the following result for the optimal linear contract.
 \begin{theorem}
 Let $\Gamma$ be symmetric positive-definite. 
The optimal level of  effort  that maximizes the linear principal's problem \eqref{eq:Vlin_d}  is  given by $\Gamma^{-1}\beta^*$ with
\begin{align}\label{eq:betalin-d}
\beta^*_t&= b^* K(T,t),  \quad t \leq T,  \\
b^* &= \frac{\langle K_T, \left(\gamma_P I_d+\Gamma^{-1}\right) K_T \rangle_{L^2}}{\langle K_T, \left(\left(\gamma_A + \gamma_P\right) I_d+\Gamma^{-1}\right) K_T \rangle_{L^2}}.
\end{align}
The utility of the principal at time $0$ is given by 
$$ V_{lin}  = V_{lin}(y_0) $$
with 
\begin{align}\label{eq:Vlin-explicit}
V_{lin}(y)= -\exp\left(-\gamma_P(g_0(T)-y) + \chi_0\right)
\end{align}
and
\begin{align*}
\chi_0 = \frac{\gamma_P}{2} \langle K_T, \left(\gamma_P I_d -b^* \left(\gamma_P I_d + \Gamma^{-1}\right)\right) K_T\rangle_{L^2}.
\end{align*}
The optimal linear contract $\xi^*$ that maximizes the linear principal's problem  is  given by 
\begin{align*}
\xi^*=y_0-b^*g_0(T)+  \int_0^T f^*(\beta^*_s) ds + b^* X_T. 
\end{align*}

\end{theorem}

\begin{proof} 
	Fix $y \geq y_0$, $b\in \R$ and $\beta_t = bK(T,t)$. Then the random variable $X_T - Y_T^{y,b}$ reads 
	\begin{align*}
X_T - Y_T^{y,b} &= g_0(T)-y   + < K_T, \left(b\Gamma^{-1} - \frac {b^2}2 \left(\gamma_A I_d + \Gamma^{-1}\right) \right)  K_T    >_{L^2}  \\
&\quad  + (1-b) \int_0^T <K(T,s),dB_s^{\beta}>
	\end{align*}
	and is therefore Gaussian under $\P^{\beta}$. So that a direct computation of the Laplace transform of a Gaussian random variable  yields 
	\begin{align*}
	\E^{\beta} \left[ U_P\left( X_T-Y_T^{y,b} \right) \right] &= - \exp\left(-\gamma_P(g_0(T)-y) + F(b) \right) 
	\end{align*}
	with 
	\begin{align*}
	F(b) = < K_T, \left(\frac{\gamma_P^2}{2} I_d -  b\left( \gamma_P^2 I_d + \gamma_P \Gamma^{-1}\right) + \gamma_P\frac {b^2}2 \left((\gamma_A + \gamma_P) I_d + \Gamma^{-1}  \right) \right)K_T   >_{L^2}.
	\end{align*}
	A direct maximization of $F$ on $b$ yields that the optimum is achieved for $b^*$ given by  \eqref{eq:betalin-d} and $F(b^*)=\chi_0$. Maximizing on $y\geq y_0$ then gives   $ V_{lin}  = V_{lin}(y_0) $. 
\end{proof}

Obviously, when the principal restricts to linear contracts,  her utility at $0$ satisfies $V_{lin}(y)  \leq V_{SB}(y) $. It follows from \eqref{eq:Vsb-explicit} and \eqref{eq:Vlin-explicit} that $\chi_0 \geq \phi_0$. More precisely, one has 
\begin{align}\label{eq:diffvalue}
V_{lin}(y_0) = V_{SB}\left( y_0 + \frac{\chi_0-\phi_0}{\gamma_P} \right) =  V_{SB}( y_0) \exp\left(\chi_0-\phi_0 \right).
\end{align}
The term $\exp\left(\phi_0-\chi_0 \right)=\frac{V_{SB}( y_0)}{V_{lin}(y_0)}$ lies in $[0,1]$ and  can be interpreted as the value of information in the following way: since in general, the principal cannot implement the optimal contract in the enlarged filtration, recall Remark~\ref{R:enlarged}, she has to restrict to sub-optimal, more simple but implementable contracts.  The price to pay when she restricts to linear contracts, is the decrease of her utility by the factor 
$\exp\left(\phi_0-\chi_0\right)$, which would correspond to the price to pay to access the optimal contract. Note that the agent utility at time $0$ is unchanged compared to the previous section, and is still equal to $\exp(\gamma_A y_0)$ by Proposition~\ref{IC-d} when  the principal proposes the contract $(y_0,b^*)$. 

\begin{remark} We note that contrary to the one dimensional setting, the coefficient $b^*$ in \eqref{eq:betalin-d} depends in general on the kernel $K$. For the case of radial costs, i.e.~$\Gamma = \kappa I_d$ for some $\kappa >0$, one recovers from \eqref{eq:betalin-d} that $b^*= (\gamma_P + {1}/{\kappa})/(\gamma_A + \gamma_P + {1}/{\kappa}) $ which is independent of $K$. Note also that in this context, $\chi_0=\phi_0$, so that the value of information vanishes in this case, that is linear contracts are optimal for the Principal's problem, recall Corollary~\ref{radial}. 
\end{remark}   

Having characterized both the fully optimal contract and the optimal linear end-of-period contract, it remains to compare the performances of the two types of contracts. We implement this comparison by studying the sensitivity of the nonnegative difference
$\chi_0-\phi_0$  with respect to the input kernel $K$ and the cost matrix $\Gamma$. The smaller the quantity $(\chi_0-\phi_0)$, the more efficient the implementation of a linear contract, recall \eqref{eq:diffvalue}. 
The next proposition provides an upper bound for the value of information in terms of two quantities: the condition number\footnote{The condition number of symmetric positive definite matrix $S$ is the ratio $\frac{\lambda_{max}}{\lambda_{min}}$ where $\lambda_{max}$ (resp. $\lambda_{min}$) is the largest (resp. smallest) eigenvalue of $S$.} of the matrix $\Gamma$, denoted $Cond(\Gamma)$ and the $L^2$-norm of the kernel $K$. The condition number $Cond(\Gamma)$ measures how sensitive is the effort cost function to changes in efforts.
\begin{pro}\label{valueinfo}
There exists a positive constant $C$ independent of the dimension $d$, kernel $K$ and terminal time $T$ such that
\begin{align}\label{eq:upperbound}
0 \leq \chi_0-\phi_0 \le C(Cond(\Gamma)-1)\int_0^T \vert \vert K(T,t) \vert\vert^2\,dt.
\end{align}
\end{pro}
\begin{proof}
		See Appendix \ref{MOP}. 
\end{proof}

When the cost is radial, then $Cond(\Gamma)=1$ so that one recovers $\chi_0 = \phi_0$, meaning that linear contracts are optimal,  recall Corollary~\ref{radial}. When $Cond(\Gamma)$ is close to one, it means that the agent's best reply in terms of effort, solution to the linear equation $\Gamma a^*=\beta$, is not very sensitive to errors in the principal control $\beta$. In that case, it is noticeable that the optimal linear contract is nearly optimal regardless of the Volterra process that drives the output.\\
For convolution kernels of the form $K(T,t)=\ind_{t<T} k(T-t)$, $\int_0^T \|K(T,t)\|^2 dt = \int_0^T \|k(t)\|^2dt $, so that the upper bound in \eqref{eq:upperbound} shrinks as the horizon of the contract $T$ decreases, suggesting that linear contract seem more performant for short-term relationships compared to long-term relationships. Furthermore, for the exponential kernel $k(t)=e^{-\lambda t}$ with $\lambda \in \R$, we have $\int_0^T \|K(T,t)\|^2 dt = (1-e^{-2\lambda T})/2\lambda$, thus the higher is the mean-reverting intensity, the smaller the upper bound.  For the fractional kernel $k(t)=\sqrt{2H}t^{H-1/2}$ with $H\in (0,1)$, we have  $\int_0^T \|K(T,t)\|^2 dt = T^{2H} $.

We now illustrate numerically the value of information $\exp(\phi_0-\chi_0)$ using Equation \eqref{vi_app} for exponential and fractional kernels with $d=2$ and with a diagonal matrix for the cost efforts $\Gamma = \mbox{diag}(\lambda_1, \lambda_2)$. First, we look at the case of two  exponential kernels $k_i(t)=e^{-\rho_i t}$ with different mean reversions\footnote{Note the change in notation to avoid confusion with the eigenvalues of $\Gamma$.} $\rho_i \in \R$, $i=1,2$.

Figure \ref{fig:exp} describes a situation where providing one unit of effort for Task 1 is more costly than for Task 2, $\lambda_1 > \lambda_2$, and when the mean-reverting parameters $\rho_i$ vary. The figure shows that linear contracts are  more performant for negative and smaller mean reversions, which is the case that is usually of interest in practice.  When the intensity of the most expensive task is fixed and positive, the variation of the intensity parameter of the least expensive task has very little effect on the value of the information. More generally, the value of information increases when $\rho_2$ increases. The linear contracts are very efficient (more than 90\%) when the mean-reverting parameter of the most expensive task is one.

\begin{center}
		\includegraphics[scale=.4]{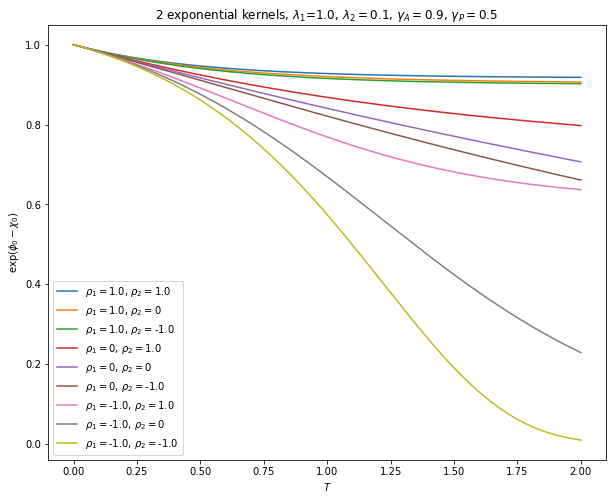}
		%	\rule{35em}{0.5pt}
		\captionof{figure}{Impact of the mean reversion parameters  $\rho_1$ and $\rho_2$ on the value of information with respect to the terminal time $T$.}
		\label{fig:exp}
\end{center}

For our second example, we consider two fractional kernels $k_i(t)=\sqrt{2H_i}t^{H_i-1/2}$ with $H_i\in (0,1)$, we see on Figure~\ref{fig:frac} that for short maturities linear contracts are more performant for  Hurst indices larger than $1/2$ (long-memory processes), while for larger maturities they are more performant for values of $H<1/2$, (short memory processes). The inflexion point at $T=1$ is explained by the behavior of the variance of the fractional Brownian motion that reads $\int_0^T \|K(T,t)\|^2 dt = T^{2H} $, see Figure~\ref{fig:L2frac}.

\begin{center}
	\includegraphics[scale=.4]{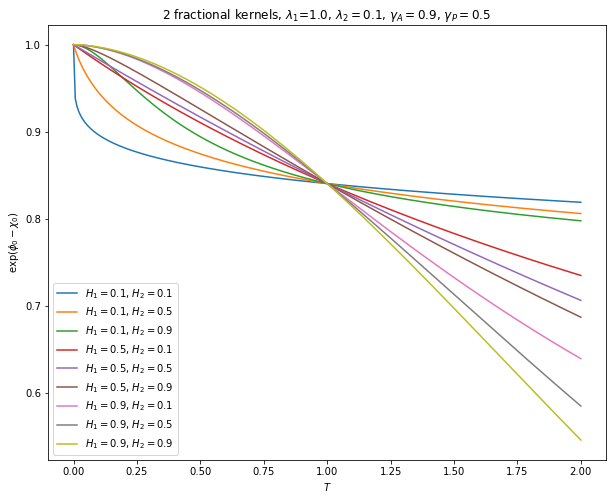}
	%	\rule{35em}{0.5pt}
	\captionof{figure}{Impact of the Hurst indices $H_1$ and $H_2$ on the value of information with respect to the terminal time $T$.}
	\label{fig:frac}
\end{center}

	\begin{center}
		\includegraphics[scale=.6]{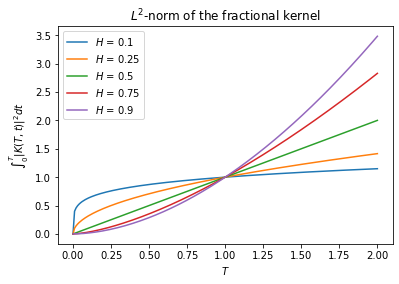}
		%	\rule{35em}{0.5pt}
		\captionof{figure}{Impact of the Hurst index $H$ on $L^2$-norm of the fractional kernel $t\mapsto\sqrt{2H}t^{H-1/2}$.}
\label{fig:L2frac}
	\end{center}

\section{Conclusion}

 The principal-agent framework generally leads to complex optimal contracts that do not align with real-world practices.  Also, the proposal of a theoretical framework justifying the signing of simple (linear) optimal contracts
  deserves attention. In this paper, we have shown that the remarkable results of the Holmstrom and Milgrom model extend surprisingly to a large class of Gaussian models that exhibit memory: the Volterra processes. In particular, we prove that optimal contracts are linear in one-dimensional models and that the principal has no incentives to expand its information set. In multi-dimensional models, this is no longer generally the case except when the effort cost function is radial. Nevertheless, we are able measure the utility gap when the principal proposes a linear contract in the case of a general effort cost function. Thus, we can examine the key features that make the performance of linear contracts very close to optimality.

\bibliography{bibl}

\begin{thebibliography}{}

\bibitem[\protect\astroncite{Abi~Jaber et~al.}{2019}]{AJLP17}
Abi~Jaber, E., Larsson, M., and Pulido, S. (2019).
\newblock Affine {V}olterra processes.
\newblock {\em The Annals of Applied Probability}, 29(5):3155--3200.

\bibitem[\protect\astroncite{Beran}{1994}]{BeranBook}
Beran, J. (1994).
\newblock Statistics for long-memory processes, volume 61 of.
\newblock {\em Monographs on Statistics and Applied Probability}, pages 38--72.

\bibitem[\protect\astroncite{Biais et~al.}{2007}]{biais2007dynamic}
Biais, B., Mariotti, T., Plantin, G., and Rochet, J.-C. (2007).
\newblock Dynamic security design: Convergence to continuous time and asset
  pricing implications.
\newblock {\em The Review of Economic Studies}, 74(2):345--390.

\bibitem[\protect\astroncite{Biais et~al.}{2010}]{biais2010large}
Biais, B., Mariotti, T., Rochet, J.-C., and Villeneuve, S. (2010).
\newblock Large risks, limited liability, and dynamic moral hazard.
\newblock {\em Econometrica}, 78(1):73--118.

\bibitem[\protect\astroncite{Bolton and Dewatripont}{2005}]{bolton2005contract}
Bolton, P. and Dewatripont, M. (2005).
\newblock {\em Contract theory}.
\newblock MIT press.

\bibitem[\protect\astroncite{Carroll}{2015}]{carroll2015robustness}
Carroll, G. (2015).
\newblock Robustness and linear contracts.
\newblock {\em American Economic Review}, 105(2):536--63.

\bibitem[\protect\astroncite{Cisternas}{2018}]{cisternas18}
Cisternas, G. (2018).
\newblock Carrer concerns and the nature of skills.
\newblock {\em American Economic Journal: Microeconomics}, 10(2):152--189.

\bibitem[\protect\astroncite{Clementi and Hopenhayn}{2006}]{clementi2006theory}
Clementi, G.~L. and Hopenhayn, H.~A. (2006).
\newblock A theory of financing constraints and firm dynamics.
\newblock {\em The Quarterly Journal of Economics}, 121(1):229--265.

\bibitem[\protect\astroncite{Cvitani{\'c} et~al.}{2018}]{cvitanic2018dynamic}
Cvitani{\'c}, J., Possama{\"\i}, D., and Touzi, N. (2018).
\newblock Dynamic programming approach to principal--agent problems.
\newblock {\em Finance and Stochastics}, 22(1):1--37.

\bibitem[\protect\astroncite{Decreusefond and
  Ustunel}{1999}]{decreusefond1999stochastic}
Decreusefond, L. and Ustunel, A.~S. (1999).
\newblock Stochastic analysis of the fractional {B}rownian motion.
\newblock {\em Potential analysis}, 10(2):177--214.

\bibitem[\protect\astroncite{DeMarzo and Fishman}{2007}]{demarzo2007optimal}
DeMarzo, P.~M. and Fishman, M.~J. (2007).
\newblock Optimal long-term financial contracting.
\newblock {\em The Review of Financial Studies}, 20(6):2079--2128.

\bibitem[\protect\astroncite{DeMarzo et~al.}{2012}]{demarzo2012dynamic}
DeMarzo, P.~M., Fishman, M.~J., He, Z., and Wang, N. (2012).
\newblock Dynamic agency and the q theory of investment.
\newblock {\em The Journal of Finance}, 67(6):2295--2340.

\bibitem[\protect\astroncite{DeMarzo and Sannikov}{2006}]{demarzo2006optimal}
DeMarzo, P.~M. and Sannikov, Y. (2006).
\newblock Optimal security design and dynamic capital structure in a
  continuous-time agency model.
\newblock {\em The Journal of Finance}, 61(6):2681--2724.

\bibitem[\protect\astroncite{Edmans and Gabaix}{2011}]{edmans2011tractability}
Edmans, A. and Gabaix, X. (2011).
\newblock Tractability in incentive contracting.
\newblock {\em The Review of Financial Studies}, 24(9):2865--2894.

\bibitem[\protect\astroncite{Edmans et~al.}{2012}]{edmans2012dynamic}
Edmans, A., Gabaix, X., Sadzik, T., and Sannikov, Y. (2012).
\newblock Dynamic ceo compensation.
\newblock {\em The Journal of Finance}, 67(5):1603--1647.

\bibitem[\protect\astroncite{Gatheral et~al.}{2018}]{gatheral2018volatility}
Gatheral, J., Jaisson, T., and Rosenbaum, M. (2018).
\newblock Volatility is rough.
\newblock {\em Quantitative Finance}, 18(6):933--949.

\bibitem[\protect\astroncite{Green}{1987}]{Green1987}
Green, E. (1987).
\newblock Lending and the smoothing of uninsurable income.
\newblock In Prescott, E. and Wallace, N., editors, {\em Contractual
  arrangements for intertemporal trade}, volume 1 of Minnesota studies in
  macroeconomics, pages 3--25. University of Minnesota Press, Minneapolis.

\bibitem[\protect\astroncite{Gripenberg et~al.}{1990}]{GLS:90}
Gripenberg, G., Londen, S.-O., and Staffans, O. (1990).
\newblock {\em Volterra integral and functional equations}, volume~34 of {\em
  Encyclopedia of Mathematics and its Applications}.
\newblock Cambridge University Press, Cambridge.

\bibitem[\protect\astroncite{Hellwig and Schmidt}{2002}]{hellwig2002discrete}
Hellwig, M.~F. and Schmidt, K.~M. (2002).
\newblock Discrete-time approximations of the holmstr{\"o}m-milgrom
  brownian-motion model of intertemporal incentive provision.
\newblock {\em Econometrica}, 70(6):2225--2264.

\bibitem[\protect\astroncite{Holmstrom and
  Milgrom}{1987}]{holmstrom1987aggregation}
Holmstrom, B. and Milgrom, P. (1987).
\newblock Aggregation and linearity in the provision of intertemporal
  incentives.
\newblock {\em Econometrica: Journal of the Econometric Society}, pages
  303--328.

\bibitem[\protect\astroncite{Hu et~al.}{2005}]{hu2005utility}
Hu, Y., Imkeller, P., and M{\"u}ller, M. (2005).
\newblock Utility maximization in incomplete markets.
\newblock {\em The Annals of Applied Probability}, 15(3):1691--1712.

\bibitem[\protect\astroncite{Karatzas and Shreve}{1991}]{karatzasshreve91}
Karatzas, I. and Shreve, S. (1991).
\newblock {\em Brownian motion and stochastic calculus}, volume 113.
\newblock Springer.

\bibitem[\protect\astroncite{Kolmogorov}{1940}]{kolmogorov1940wienersche}
Kolmogorov, A.~N. (1940).
\newblock Wienersche spiralen und einige andere interessante kurven in
  hilbertscen raum, cr (doklady).
\newblock {\em Acad. Sci. URSS (NS)}, 26:115--118.

\bibitem[\protect\astroncite{Laffont and Martimort}{2009}]{laffont2009theory}
Laffont, J.-J. and Martimort, D. (2009).
\newblock {\em The theory of incentives}.
\newblock Princeton university press.

\bibitem[\protect\astroncite{Mandelbrot and
  Van~Ness}{1968}]{mandelbrot1968fractional}
Mandelbrot, B.~B. and Van~Ness, J.~W. (1968).
\newblock Fractional brownian motions, fractional noises and applications.
\newblock {\em SIAM review}, 10(4):422--437.

\bibitem[\protect\astroncite{Rogerson}{1985}]{rogerson1985repeated}
Rogerson, W.~P. (1985).
\newblock Repeated moral hazard.
\newblock {\em Econometrica: Journal of the Econometric Society}, pages 69--76.

\bibitem[\protect\astroncite{Sannikov}{2008}]{sannikov2008continuous}
Sannikov, Y. (2008).
\newblock A continuous-time version of the principal-agent problem.
\newblock {\em The Review of Economic Studies}, 75(3):957--984.

\bibitem[\protect\astroncite{Sch{\"a}ttler and Sung}{1993}]{schattler1993first}
Sch{\"a}ttler, H. and Sung, J. (1993).
\newblock The first-order approach to the continuous-time principal-agent
  problem with exponential utility.
\newblock {\em Journal of Economic Theory}, 61(2):331--371.

\bibitem[\protect\astroncite{Spear and Srivastava}{1987}]{spear1987repeated}
Spear, S.~E. and Srivastava, S. (1987).
\newblock On repeated moral hazard with discounting.
\newblock {\em The Review of Economic Studies}, 54(4):599--617.

\bibitem[\protect\astroncite{Sung}{1995}]{sung1995linearity}
Sung, J. (1995).
\newblock Linearity with project selection and controllable diffusion rate in
  continuous-time principal-agent problems.
\newblock {\em The RAND Journal of Economics}, pages 720--743.

\bibitem[\protect\astroncite{Thomas and Worrall}{1990}]{thomas1990income}
Thomas, J. and Worrall, T. (1990).
\newblock Income fluctuation and asymmetric information: An example of a
  repeated principal-agent problem.
\newblock {\em Journal of Economic Theory}, 51(2):367--390.

\bibitem[\protect\astroncite{Williams}{2011}]{williams2011ECMA}
Williams, N. (2011).
\newblock Persistent private information.
\newblock {\em Econometrica: Journal of the Econometric Society},
  79(4):1233--1275.

\end{thebibliography}
\bibliographystyle{apa}
\addcontentsline{toc}{section}{References}
\bibliographystyle{plain}

\section{Appendix}

\subsection{Proofs} \label{MOP}

\begin{proof} [Proof of Lemma~\ref{martingaleoptimality}]
Let us consider any admissible effort policy $a$ in $\mathcal{A}$. Conditions i)-iv) immediately imply that
\begin{align*}
&\E^a\left[U_A\left(\xi-\int_0^{T} k(a_s) ds\right)\right]\\
&\overset{i)}{=}\E^a[R_T^a] \\
&\overset{ii)}{\leq} R_0^a \\
&\overset{iii)}{=} R_0^{a^\ast} \\
&\overset{iv)}{=} \E^{a^\ast}[R_T^{a^\ast}] \\
&\overset{i)}{=} \E^{a^\ast}\left[U_A\left(\xi-\int_0^{T} k(a_s^\ast) ds\right)\right]
\end{align*}
which concludes the proof. 
 \end{proof} 
 
\begin{proof}[Proof of Lemma~\ref{L:M}]
Denoting by 	$  U^{\beta}_t = -\gamma_P \left( g^{\beta}_t(T) - Y_t^{\beta}  \right) 	+ \phi_t,  $
an application of It\^o's formula yields 
$$ d M_t^{\beta} =  M_t^{\beta} \left(dU_t^{\beta}  + \frac 1 2 d\langle U^{\beta}\rangle_t\right).$$
Using \eqref{eq:dynamicsg}, we obtain that
$$ dU_t^{\beta} =  \left(\dot{\phi}_t -\gamma_P K(T,t) \kappa\beta_t  + \gamma_P\frac{\gamma_A+1/ \kappa}{2}   \beta_t^2 \right)dt +   \left(\gamma_P \beta_t -\gamma_P K(T,t) \right)    dB^{\beta}_t, \quad \mathbb P^{\beta}-a.s.$$
so that 
	\begin{align}
\frac{dM_t^{\beta}}{M_t^{\beta}} &=   \left(\dot{\phi}_t  + \frac 1 2 \gamma_P^2 K(T,t)^2    + \frac{\gamma_P} 2 (\gamma_A  +\gamma_P+1/ \kappa) \beta_t^2  - \gamma_P (\gamma_P + 1/ \kappa) K(T,t)\beta_t  \right)dt  \\
&\quad +  \left(\gamma_P \beta_t -\gamma_P K(T,t) \right)    dB^{\beta}_t , \quad \mathbb P^{\beta}-a.s.  
\end{align}
Completing the squares in $\beta$ yields 
\begin{align*}
 \frac{\gamma_P} 2 (\gamma_A  +\gamma_P+1/ \kappa) \beta_t^2  - \gamma_P (\gamma_P + 1/ \kappa) K(T,t)\beta_t 
 &=  \frac{\gamma_P} 2 (\gamma_A^2  +\gamma_P+1/ \kappa) \left(\beta_t - \beta^*_t \right)^2   \\
 &\quad \quad \quad  - \frac{\gamma_P} 2  \frac{(\gamma_P + 1/ \kappa)^2}{(\gamma_A  +\gamma_P+1/\kappa)} K(T,t)^2 ,
\end{align*}
with $\beta^*$ given by \eqref{eq:Maineffort}. Combining the above and using that 
$$  \dot{\phi}_t = \frac{\gamma_P} 2 \left(  \frac{(\gamma_P + 1/ \kappa)^2}{(\gamma_A +\gamma_P+1/ \kappa)} -   \gamma_P^2  \right) K(T,t)^2 $$
yields \eqref{eq:Mdynamics}. 
\end{proof}

\begin{proof} [Proof of Theorem~\ref{Main-d}] The incentive-compatible contract parametrized by $(y,\beta)$  takes the form
	$$ dY_t^{y,\beta} =  \frac{1}{2} <\beta_t, D \beta_t> dt + <\beta_t,  dB_t^{\beta}>, \quad Y_0^{\beta} = y,$$
	with $D:=\gamma_A I_d + \Gamma^{-1}$.
	Under the probability $\P^\beta$, the output process evolves as
	$$ X_t^{\beta} = g_0(t) + \int_0^t <K(t,s), \Gamma^{-1}\beta_s> ds + \int_0^t <K(t,s), dB_s^{\beta}>.$$
	For $ s\geq t$, define the effort-corrected forward output by
	\begin{align*}
	g_t^{\beta}(s) &= \E^{\beta}\left[ X_s^{\beta} - \int_t^s <K(s,u),\Gamma^{-1}\beta_u> du \mid \mathcal F_t\right]\\
	&= g_0(s) + \int_0^t <K(s,u), \Gamma^{-1}\beta_u> du + \int_0^t <K(s,u), dB_u^{\beta}>.
	\end{align*}
	Note that for fixed $s \leq T$, $t\mapsto g^{\beta}_t(s)$ is a semimartingale on  $[0,s)$ with dynamics 
	\begin{align}\label{eq:dynamicsgdd}
	dg^{\beta}_t(s) = <K(s,t),\Gamma^{-1}\beta_t> dt + <K(s,t),dB_t^{\beta}>.
	\end{align}
	To mimic the proof of the one-dimensional case, we will apply the martingale optimality principle to the process
	\begin{align*} M_t^{\beta} = \exp\left(-\gamma_P \left( g^{\beta}_t(T) - Y_t	^{\beta} \right)  + \phi_t   \right),
	\end{align*} 
	with 
	\begin{align}\label{eq:phidd}
	\phi_t  =  \frac{\gamma_P} 2  \int_t^T < K(T,s),    \left( \gamma_P I_d   - \left(\gamma_PI_d  + \Gamma^{-1}\right) \left(\gamma_PI_d  + D\right)^{-1}  \left(\gamma_PI_d  + \Gamma^{-1}\right)  \right)     K(T,s)>  ds .
	\end{align}
	Denoting by 	$$  U^{\beta}_t = -\gamma_P \left( g^{\beta}_t(T) - Y_t^{\beta}  \right) 	+ \phi_t ,  $$
	an application of It\^o's formula yields 
	$$ d M_t^{\beta} =  M_t^{\beta} \left(dU_t^{\beta}  + \frac 1 2 d\langle U^{\beta}\rangle_t\right).$$
	Using \eqref{eq:dynamicsgdd}, we obtain that
	\begin{align*}
	dU_t^{\beta} &=  \left(\dot{\phi}_t -\gamma_P <K(T,t), \Gamma^{-1}\beta_t>  + \frac{\gamma_P} {2} <\beta_t,D\beta_t>  \right)dt \\
	&\quad +  \gamma_P <\beta_t  - K(T,t),   dB^{\beta}_t>, \quad \mathbb P^{\beta}-a.s.
	\end{align*} 
	so that 
	\begin{align}
	\frac{dM_t^{\beta}}{M_t^{\beta}} &=   \bigg[\dot{\phi}_t  + \frac 1 2 \gamma_P^2 \|K(T,t)\|^2 + \frac{\gamma_P} {2} <\beta_t, \left(\gamma_P  I_d +  D\right) \beta_t>   \\
	&\quad \quad - \gamma_P  <K(T,t), 	(\gamma_PI_d  + \Gamma^{-1})\beta_t>    \bigg]dt  \\
	&\quad +  \gamma_P <\beta_t  - K(T,t)  , dB^{\beta}_t>, \quad \mathbb P^{\beta}-a.s.  
	\end{align}
	Completing the squares in $\beta$ yields that 
	\begin{align*}
\frac{\gamma_P} {2} <\beta_t, \left(\gamma_P  I_d +  D\right) \beta_t>   - \gamma_P  <K(T,t), 	(\gamma_PI_d  + \Gamma^{-1})\beta_t>  
	\end{align*}
	is equal to 
	\begin{align*}
&\frac{\gamma_P} {2} <\left(\beta_t- \beta_t^*\right), \left( \gamma_P I_d + D \right)  \left(\beta_t- \beta_t^*\right)> \\
&\quad \quad \quad -  \frac{\gamma_P} {2}  <K(T,t), \left(\gamma_P I_d + \Gamma^{-1}\right) \left(\gamma_P I_d + D\right)^{-1}\left(\gamma_P I_d + \Gamma^{-1}\right) K(T,t)>	
	\end{align*}
	with $\beta^*$ given by \eqref{eq:betavsb-d}. 
	Combining the above computations and using \eqref{eq:phidd}
	yield 
	\begin{align}\label{eq:Mdynamicsd}
	\frac{dM_t^{\beta}}{M_t^{\beta}} &=\frac{\gamma_P} {2} <\left(\beta_t- \beta_t^*\right), \left( \gamma_P I_d + D \right)  \left(\beta_t- \beta_t^*\right)> dt \\
	&\quad +  \gamma_P < \left(\beta_t  -K(T,t) \right), dB_t^{\beta}>, \quad \mathbb P^{\beta}-a.s.  
	\end{align}
	We then conclude the proof by proceeding analogously to the proof of Theorem~\ref{Main} in Section~\ref{sectionP} in the one-dimensional case by applying the Martingale optimality principle on $M^{\beta}$.
\end{proof}

\begin{proof} [Proof of Proposition~\ref{valueinfo}]
Let us define the symmetric positive definite matrix $$A=\gamma_P I_d +\Gamma^{-1}.$$
 The matrix A is diagonalizable and thus, there exists an orthogonal matrix $P$ such that 
 $$^tPAP=D=diag(\eta_1,\ldots,\eta_d)$$ with
$\eta_i=\gamma_P+\frac{1}{\lambda_i}$ where $\lambda_1,\ldots,\lambda_d$ are the eigenvalues of $\Gamma$.
We have
\begin{align}
\chi_0-\phi_0&=  <K_T,A(\gamma_A I_d+A)^{-1}AK>-\frac{<K_T,AK_T>^2}{<K_T,(\gamma_AI_d+A)K_T>}.
\end{align}
Setting $K_T(t)=P\hat K_T(t)$ for every $t \le T$ and denoting by $\hat k_i(t)$ the components of the vector $\hat K_t(t)$, we obtain 

\begin{align*}
\chi_0-\phi_0 =<\hat K_T(t),D(\gamma_A I_d+D)^{-1}D \hat K_T(t)>-\frac{<\hat K_T(t),D \hat K_T(t)>^2}{<\hat K_T(t),(\gamma_AI_d+D)\hat K_T(t)>},
\end{align*}
or equivalently
\begin{align}\label{vi_app}
\chi_0-\phi_0 =\sum_{i=1}^d \frac{\eta_i^2}{\gamma_A+\eta_i} \int_0^T \hat k_i(t)^2 dt  -  \frac{\left( \sum_{i=1}^d \eta_i \int_0^T \hat k_i(t)^2 dt  \right)^2 }{\sum_{i=1}^d(\gamma_A+\eta_i) \int_0^T \hat k_i(t)^2 dt }.
\end{align}

Then, observing that 
$$ \sum_{i=1}^d  \int_0^T \hat k_i(t)^2 dt = \int_0^T  \|^tP K(T,t)\|^2 dt  = \int_0^T \|K(T,t)\|^2 dt,  $$ 
since $P$ is orthogonal,  we deduce that 
\begin{align*}
\chi_0-\phi_0 \le  \left( \frac{\eta_{max}^2}{\gamma_A+\eta_{min}}-\frac{\eta_{min}^2}{\gamma_A+\eta_{max}} \right) \int_0^T  \|K(T,t)\|^2 dt 
\end{align*}
Rearranging terms and using $a^3-b^3=(a-b)(a^2+ab+b^2)$, we finally obtain
$$
\chi_0-\phi_0 \le C(\eta_{max}-\eta_{min}) \int_0^T  \|K(T,t)\|^2 dt ,
$$
from which we deduce the result by noting that
$$
\eta_{max}-\eta_{min}=\frac{1}{\lambda_{max}}(Cond(\Gamma)-1).
$$

\end{proof}

\subsection{Stochastic linear integro-differential convolution equations} \label{Integro}

In this section, we consider the stochastic linear integro-differential convolution equations of the form 
\begin{align}\label{eq:stochastic integro}
dX_t = \left( h(t) + \int_{[0,t]} \mu(ds) X_{t-s} \right) dt + \sigma dB_t , 
\end{align}
with initial condition $X_0 \in \R^d$, where $h:[0,T] \to \R^{d}$, $\mu:\mathcal B([0,T]) \to \R^{d\times d}$ of bounded variation, $X_0 \in \R^d$ and $\sigma \in \R^{d\times d}$. We show that \eqref{eq:stochastic integro} admits a unique solution given by the Volterra Gaussian process 
\begin{align}\label{eq:volterra}
X_t = g_0(t) + \int_0^t K(t,s) dB_s 
\end{align} 
for some specific choice of input curve $g_0:[0,T]\to \R^d$ and convolution kernel $K:[0,T]^2 \to \R^{d\times d}$.

\begin{example}
	Setting $\mu(dt) = \sum_{k=1}^m a_k \delta_{t_k} $, we recover equations with delay.	
\end{example}

For a Lebesgue measurable matrix-valued function $f$ and a matrix-valued measure $\mu$ of locally bounded variation  we define the convolutions $f*\mu$ and $\mu *f $ by 
$$ (f \ast \mu)(t) = \int_{[0,t]}  f(t-s)\mu(ds), \quad (\mu \ast f)(t) = \int_{[0,t]} \mu(ds) f(t-s).$$
We note that  $f*\mu = \mu*f$ when $d=1$.

We need the notion of the {differential resolvent} of $\mu$. For any bounded measure with locally finite variation $\mu$, the {differential resolvent} $R:[0,T]\to \R^{d\times d}$ is the unique locally absolutely continuous function $R$ such that 
\begin{align}\label{eq:resolventeq}
R'= \mu*R = R*\mu, \quad R(0)=I_d. 
\end{align} 
We refer to \cite[Theorem 3.1]{GLS:90} for the existence and uniqueness statement regarding $R$. One observes  that $R$ is continuous and therefore bounded on $[0,T]$ so that the stochastic convolution $\int_0^{\cdot} R(\cdot-s)\sigma dB_s$ is well-dedined on $[0,T]$.

\begin{theorem}
	Let $h:[0,T] \to \R^{d}$, $\mu:\mathcal B([0,T]) \to \R^{d\times d}$ of bounded variation and $X_0 \in \R^d$. The stochastic linear integrodifferential equation \eqref{eq:stochastic integro} admits a unique (continuous) solution on $[0,T]$ given by  
	\begin{align}\label{eq:sol}
	X_t = g_0(t) + \int_0^tR(t-s) \sigma dB_s, 
	\end{align}
	with 
	$$ g_0(t) = R(t)X_0 + \int_0^t R(t-s)h(s)ds,$$
	and $R$ the differential resolvent of $\mu$.
\end{theorem}

\begin{proof}[Sketch of proof]
	
	$\bullet$ Uniqueness: we show that any solution $X$ to \eqref{eq:stochastic integro} is given by \eqref{eq:sol}. We first write \eqref{eq:stochastic integro} in compact integral form
	$$ X = X_0+ 1*h  + 1*(\mu*X) + 1*\sigma dB.$$
 Then,  convolving both sides with $R$, and applying stochastic Fubini's theorem combined with the resolvent equation \eqref{eq:resolventeq},  leads to 
	\begin{align*}
	R*X&= R*X_0 + R*(1*h)+ R*(1*\mu*X)  + R*(1*\sigma dB)\\
	&= 1*(RX_0) + 1*(R*h) + (1*R*\mu)*X + 1*(R*\sigma dB)\\
	&= 1*(RX_0) + 1*(R*h)+ R*X - R(0)(1*X) + 1*(R*\sigma dB)
	\end{align*}
	Simplifying $R*X$ on both sides, recalling that $R(0)=I_d$, and inspecting the densities lead to 
	$$ X_t= R(t)X_0 + \int_0^t R(t-s)h(s)ds + \int_0^t R(t-s)\sigma dB_s , \quad dt \times d\P-a.e.  $$
	The conclusion follows from the  continuity of the sample paths.

	$\bullet$ Existence: we verify that $X$  given by \eqref{eq:sol} solves \eqref{eq:stochastic integro}.  We use the resolvent equation to write $R = I_d +  1*(\mu*R)$ so that  \eqref{eq:sol} reads 
	$$ X = X_0 +  1*\left((\mu*R) X_0   +  h +  (\mu*R)*h +  (\mu*R)*\sigma dB  \right) + 1*\sigma dB, $$
	which shows that $X$ is a continuous semimartingale with the following dynamics 
	\begin{align*}
	dX_t &= \left( h (t)+ \left(\mu*\left(R X_0   +   R*h +  R*\sigma dB\right)\right)(t) \right) dt + \sigma dB_t \\
	&=\left( h (t)+ \left(\mu*X\right)(t) \right) dt + \sigma dB_t
	\end{align*}
	where we made use of stochastic Fubini's theorem. The proof is complete. 
\end{proof}

% added by arXiv:
\typeout{get arXiv to do 4 passes: Label(s) may have changed. Rerun}

\end{document}